\documentclass[reqno]{amsart}

\usepackage{amssymb}
\usepackage[usenames, dvipsnames]{xcolor}
\usepackage{verbatim}

\numberwithin{equation}{section}

\newtheorem{theorem}{Theorem}[section]
\newtheorem{corollary}[theorem]{Corollary}
\newtheorem{lemma}[theorem]{Lemma}
\newtheorem{proposition}[theorem]{Proposition}

\theoremstyle{definition}
\newtheorem{remark}[theorem]{Remark}

\theoremstyle{definition}

\theoremstyle{definition}
\newtheorem{assumption}[theorem]{Assumption}

\def\vu{\textit{\textbf{u}}}

\def\bA{\mathbb{A}}
\def\bR{\mathbb{R}}

\def\bZ{\mathbb{Z}}

\def\bH{\mathbb{H}}

\def\bC{\mathbb{C}}
\def\bL{\mathbb{L}}
\def\bO{\mathbb{O}}

\def\fB{\mathfrak{B}}

\def\fQ{\mathfrak{Q}}

\def\cC{\mathcal{C}}
\def\cD{\mathcal{D}}

\def\cH{\mathcal{H}}

\def\cM{\mathcal{M}}
\def\cO{\mathcal{O}}

\def\cT{\mathcal{T}}
\def\cL{\mathcal{L}}

\newcommand{\Div}{\operatorname{div}}

\makeatletter
\def\dashint{\operatorname%
{\,\,\text{\bf--}\kern-.98em\DOTSI\intop\ilimits@\!\!}}
\makeatother

\begin{document}
\title[Parabolic equations in simple convex polytopes]{Parabolic equations in simple convex polytopes with time irregular coefficients}

\author[H. Dong]{Hongjie Dong}
\address[H. Dong]{Division of Applied Mathematics, Brown University,
182 George Street, Providence, RI 02912, USA}
\email{Hongjie\_Dong@brown.edu}
\thanks{H. Dong was partially supported by the NSF under agreement DMS-1056737.}

\author[D. Kim]{Doyoon Kim}
\address[D. Kim]{Department of Applied Mathematics, Kyung Hee University, 1732 Deogyeong-daero, Giheung-gu, Yongin-si, Gyeonggi-do 446-701, Republic of Korea}
\email{doyoonkim@khu.ac.kr}
\thanks{D. Kim was supported by Basic Science Research Program through the National Research Foundation of Korea (NRF) funded by the Ministry of Education, Science and Technology (2011-0013960).}



\begin{abstract}
We prove the $W^{1,2}_p$-estimate and solvability for the Dirichlet
problem of second-order parabolic equations in simple convex polytopes with time irregular coefficients, when $p\in (1,2]$.
We also consider the corresponding Neumann problem in a half space when $p\in [2,\infty)$. Similar results are obtained for equations in a half space with coefficients which are measurable in a tangential direction and have small mean oscillations in the other
directions.
\end{abstract}

\maketitle

\section{Introduction}        \label{sec1}

This paper is concerned with the $L_p$-theory of second-order linear parabolic equations in non-divergence form with discontinuous coefficients in simple convex polytopes. By a simple polytope we mean a $d$-dimensional polytope (not necessarily bounded) each of whose vertices, if there are any, is adjacent to exactly $d$ edges. This work is a continuation of \cite{Dong12b}, in which the first author studied elliptic equations in a half space or in a 2D convex wedge. Equations with discontinuous coefficients in non-smooth domains emerge from problems in mechanics, engineering, and biology, to name a few.

The $L_p$-theory of non-divergence form second-order elliptic and
parabolic equations in the whole space or smooth domains with discontinuous coefficients was studied extensively in the last fifty years. According to the well-known counterexamples of Ural'tseva \cite{Ur67} and Nadirashvili \cite{N97}, in general there does not exist a solvability theory for uniformly elliptic operators with general bounded and measurable coefficients. Many efforts have been made to treat particular types of discontinuous coefficients. See, for instance, \cite{BN54, Ca67, Kr70, Lo72a, Lo72b, Sa76, Chi} and recent work \cite{Kr07, Kr07_mixed, Ki07, KK07}. In these papers, either the leading coefficients are assumed to be measurable with respect to one or two variables and sufficiently regular with respect to the other variables, or $p$ is in a small neighborhood of $2$.

There is also a vast literature on the $L_p$-theory for elliptic and parabolic equations with smooth coefficients in domains with wedges or with conical or angular points. See, for instance, \cite{Ko67,Gr85, KO83,Da88,NP94, KMR97,MR03, MR08} and the references therein. In \cite{Lo75,Lo77a} Lorenzi considered elliptic equations with piecewise constant coefficients in two sub-angles of an angular domain in the plane with a zero right-hand side and inhomogeneous Dirichlet boundary conditions. Solvability results in weighted Sobolev spaces were established for the heat equation in a dihedral angle by Solonnikov \cite{So84, So01}, and in a wedge with edge of an arbitrary codimension by Nazarov \cite{Na01} under the assumption that the base of the wedge has a smooth boundary. The proofs in these two papers are based on the estimates of the corresponding Green's functions. Very recently Kozlov and Nazarov \cite{KN11} extended the result in \cite{Na01} to parabolic equations with leading coefficients which are measurable functions with respect to the time variable. This result is in the same spirit as Lieberman \cite{Li92} and Krylov \cite{Kr01, Kr07}, in which it is shown that for certain Schauder and $L_p$-estimates of parabolic equations one does not need any regularity assumptions on the coefficients with respect to the time variable. We also mention that the Dirichlet and Neumann boundary value problems for the heat equation in bounded Lipschitz domains were studied by Wood \cite{Wo07}. For divergence type parabolic equations in irregular domains, we refer the reader to a recent paper \cite{AK11} and the references therein.

The objective of this paper is to study $L_p$-estimates for parabolic equations in simple convex polytopes. For $p\in (1,2]$, we prove the $W^{1,2}_p$-estimate and solvability for the Dirichlet problem with most leading coefficients merely measurable with respect to the time variable and one spacial variable (Theorem \ref{thm01}). This range of $p$ is sharp even for Laplace equations in polygons. See, for instance, \cite[Theorem 4.3.2.4]{Gr85}, or \cite[Sect. 4.3.1]{MR08}. We also consider the corresponding Neumann problem in a half space when $p\in [2,\infty)$ (Theorem \ref{thm02}). Similar results are obtained for equations in a half space with coefficients which are measurable in a tangential direction and have small mean oscillations in the other directions
(Theorems \ref{thm3} and \ref{thm4}). At a conceptual level, Theorem \ref{thm01} appears to be close to the main result in \cite{KN11} mentioned above. However, the base of the wedge which we treat is non-smooth and is not covered by the results in \cite{KN11}. On the other hand, we only consider estimates in Sobolev spaces without weights. At a technical level, our arguments are completely different from those in \cite{KN11}.

For the proofs, we note that the classical Calder\'on--Zygmund approach cannot be applied to our problem due to the lack of regularity of the coefficients and the domain. Our proofs are motivated by Krylov \cite{Kr07}, in which the author presented a unified
approach to investigating the $L_p$-solvability of both
divergence and non-divergence form  parabolic (and elliptic)  equations in the whole space when the leading coefficients have vanishing mean oscillations (VMO) in the spatial variables and are measurable in the time variable. However, this approach is not directly applicable here by the same reason above. Roughly speaking, our main idea of the proofs is that after a suitable change of variables we can rewrite the operator into a divergence form operator of a certain type. With the Dirichlet boundary condition, we prove that certain first-order derivatives of the solution satisfy divergence form equations with the conormal derivative boundary condition\footnote{This crucial observation was used before by Jensen \cite{Je80} and Lieberman \cite{Li92} in different contexts.}. While with the Neumann boundary condition, the normal derivative of the solution satisfies a divergence form equation of a special type with the Dirichlet boundary condition. Therefore, we reduce the problem to certain $L_p$-estimates for these divergence form
operators, for which the interior and boundary
$C^{\alpha}$-estimates of certain first derivatives are available due to the De Giorgi--Nash--Moser estimate. Then we are able to use Krylov's approach mentioned above to establish the desired $L_p$-estimates.

The remaining part of the paper is organized as follows. We introduce some notation and state our main theorems, Theorems \ref{thm01} and \ref{thm02}, in the next section. Section \ref{secNonDiv2} is devoted to the proof of Theorem \ref{thm01} in the special case when $p=2$, with more general boundary conditions. To prove the general case, in Section \ref{sec3} we consider divergence form parabolic operators of two different types. We obtain $L_p$-estimates for these operators, which are crucial in the proofs of the main theorems. We complete the proofs of Theorems \ref{thm01} and \ref{thm02} in Section \ref{sec4}. In Section \ref{sec6}, we treat parabolic equations in a half space with coefficients which are measurable with respect to a tangential direction of the boundary of the half space and VMO with respect to the other variables. 

\section{Main results}\label{mains}

First let us fix some notation. For $r>0$ and $(t_0,x_0)\in \bR^{d+1}$, we write
$$
B_r(x_0) = \{ x \in \bR^d : |x-x_0| < r \},
$$
$$
Q_r(t_0,x_0) = \{ (t,x) \in \bR^{d+1}: t_0-r^2 < t < t_0, |x-x_0| < r \}.
$$
For any integer $1\le k\le d$, we define a wedge with edges of codimension $d-k$
$$
\Xi^k=\Xi^k_d:=\{x\in \bR^d\,:\,x_1,\ldots,x_k\in \bR^+\},
$$
which is a special type of simple convex polytopes.
Its boundary $\partial \Xi^k$ is composed of $k$ faces
$$
\Gamma^{k,i}:=\{x\in \bR^d\,:\,x_i=0\}\cap\partial \Xi^k,\quad 1\le i\le k.
$$
Note that $\Xi^1=\bR^d_+=\{x = (x_1,\ldots,x_d) \in \bR^d: x_1 > 0 \}$ and $\Gamma^{1,1}=\partial\bR^d_+$.
We set
$$
B_r^{k}(x_0)=B_r(x_0) \cap \Xi^k,\quad
Q_r^{k}(t_0,x_0)=(t_0-r^2,t_0)\times B_r^k(x_0),
$$
$$
\Gamma_r^{k,i}(x_0)=B_r(x_0) \cap \Gamma^{k,i},\quad 1\le i\le k.
$$
We write, for example, $B_r^k$ if $x_0 = 0$ and $Q_r^k$ if $(t_0,x_0) = (0,0)$.
If $\cD \subset \bR^d$ or $\cD \subset \bR^{d+1}$, where $\cD$ is not necessarily open in $\bR^d$ or $\bR^{d+1}$, we denote $\varphi \in C_0^{\infty}(\cD)$ if $\varphi$ is infinitely differentiable on $\cD$ and its compact support belongs to $\cO \cap \cD$, where $\cO$ is an open set in $\bR^d$ or $\bR^{d+1}$.
For the average of $f$ over $\cD \subset \bR^{d+1}$, we use the notation
$$
\dashint_{\cD} f(t,x) \, dx \, dt := \frac{1}{|\cD|} \int_{\cD} f(t,x)  \, dx \, dt ,
$$
where $|\cD|$ is the $d+1$-dimensional Lebesgue measure of $\cD$.

We assume that the operator is uniformly non-degenerate, i.e., there exists $\delta \in (0,1)$ such that
\begin{equation}
                                    \label{eq5.21}
\delta|\xi|^2 \le \sum_{i,j=1}^d a^{ij}\xi_i \xi_j,
\quad |a^{ij}| \le \delta^{-1}
\end{equation}
for all $\xi \in \bR^d$.
In Theorems \ref{thm01} and \ref{thm02} below we assume that the coefficients $a^{ij}$ are measurable functions of $(t,x_d) \in \bR^2$ except $a^{dd}$ which is a measurable function of either $t$ or $x_d$. That is,
\begin{equation}
							\label{eq5.22}
\begin{aligned}
a^{ij}=a^{ij}(t,x_d)
\quad
\text{if}
\quad
(i,j) \ne (d,d).
\\
a^{dd} = a^{dd}(t)
\quad
\text{or}
\quad
a^{dd} = a^{dd}(x_d).
\end{aligned}
\end{equation}

Throughout the paper, we set $\bR_T := (-\infty,T)$, where $T \in (-\infty,\infty]$.
Let $W_p^{1,2}(\bR_T \times \Xi^k)$ be the solution spaces for non-divergence type parabolic equations defined by
$$
W_p^{1,2}(\bR_T \times \Xi^k)
= \{ u, Du, D^2u, u_t \in L_p(\bR_T \times \Xi^k)\},
$$
\begin{multline*}
\|u\|_{W_p^{1,2}(\bR_T \times \Xi^k)}
= \|u\|_{L_p(\bR_T \times \Xi^k)}
+ \|Du\|_{L_p(\bR_T \times \Xi^k)}
\\
+ \|D^2u\|_{L_p(\bR_T \times \Xi^k)}
+ \|u_t\|_{L_p((\bR_T \times \Xi^k)}.
\end{multline*}

Now we state the main results of the paper. Our first result is about the Dirichlet problem in the domain $\bR_T\times \Xi^k$.

\begin{theorem}[The Dirichlet problem]
							\label{thm01}
Let $p \in (1,2]$, $T \in (-\infty,\infty]$, $\lambda \ge 0$, $k \in [1,d]$ be an integer, and $f\in L_p(\bR_T \times \Xi^k)$.
Assume that $a^{ij}$ satisfy \eqref{eq5.21} and \eqref{eq5.22}.
If $u \in W_p^{1,2}(\bR_T \times \Xi^k)$ satisfies $u = 0$ on $\bR_T \times \partial \Xi^k$ and
\begin{equation}
								\label{eq01}
- u_t + a^{ij} D_{ij}u - \lambda u = f
\end{equation}
in $\bR_T \times \Xi^k$,
then we have
\begin{multline}
							\label{eq1201}
\lambda \|u\|_{L_p(\bR_T \times \Xi^k)} + \lambda^{1/2}\|Du\|_{L_p((\bR_T \times \Xi^k)} + \|D^2u\|_{L_p(\bR_T \times \Xi^k)}
+ \|u_t\|_{L_p(\bR_T \times \Xi^k)}
\\
\le N \|f\|_{L_p(\bR_T \times \Xi^k)},
\end{multline}
where $N=N(d,\delta,p)>0$.
Moreover, for any $f \in L_p(\bR_T \times \Xi^k)$ and $\lambda > 0$, there exists a unique $u \in W_p^{1,2}(\bR_T \times \Xi^k)$ satisfying \eqref{eq01} with the Dirichlet boundary condition $u=0$ on $\bR_T \times \partial \Xi^k$.
\end{theorem}

\begin{remark}
                                \label{rem2.2}
By using the odd/even extensions with respect to $x_d$, without loss of generality, in Theorem \ref{thm01} we may assume that $k<d$. Indeed, if $k=d$, we set $\bar{u}$ and $\bar{f}$ to be the odd extensions of $u$ and $f$ with respect to $x_d$, respectively. Then we have
$$
\|\bar{u}\|_{W_p^{1,2}(\bR_T \times \Xi^{d-1})} \cong \|u\|_{W_p^{1,2}(\bR_T \times \Xi^d)},
\quad
\|\bar f\|_{L_p(\bR_T \times \Xi^{d-1})} \cong \|f\|_{L_p(\bR_T \times \Xi^d)}.
$$
Moreover, $\bar{u} \in W_p^{1,2}(\bR_T \times\Xi^{d-1})$ satisfies
\begin{equation*}
-\bar{u}_t + \bar a^{ij} D_{ij} \bar{u} - \lambda \bar{u} = \bar{f}
\end{equation*}
in $\bR_T \times \Xi^{d-1}$ with the Dirichlet boundary condition on $\bR_T \times \partial \Xi^{d-1}$, where $\bar a^{dj}$ and $\bar a^{jd}$, $j=1,\ldots,d-1$,  are the odd extensions of $a^{dj}$ and $a^{jd}$ with respect to $x_d$, and  $\bar a^{ij}$  are the even extensions of $a^{ij}$ for all the other indices $(i,j)$. It is clear that $\bar a^{ij}$ satisfy \eqref{eq5.21} and \eqref{eq5.22}. Therefore, the case $k=d$ follows from the case $k=d-1$.
\end{remark}

The second result is regarding the Neumann problem in a half space.

\begin{theorem}[The Neumann problem]
							\label{thm02}
Let $p \in [2,\infty)$, $T \in (-\infty,\infty]$, $\lambda \ge 0$, and $f\in L_p(\bR_T \times \bR^d_+)$.
Assume that $a^{ij}$ satisfy \eqref{eq5.21} and \eqref{eq5.22}.
If $u \in W_p^{1,2}(\bR_T \times \bR^d_+)$ satisfies $D_1u = 0$ on $\bR_T \times \partial \bR^d_+$ and
\begin{equation}
								\label{eq1203}
- u_t + a^{ij} D_{ij}u - \lambda u = f
\end{equation}
in $\bR_T \times \bR^d_+$,
then we have
\begin{multline*}
\lambda \|u\|_{L_p(\bR_T \times \bR^d_+)} + \lambda^{1/2}\|Du\|_{L_p(\bR_T \times \bR^d_+)} + \|D^2u\|_{L_p(\bR_T \times \bR^d_+)}
+ \|u_t\|_{L_p(\bR_T \times \bR^d_+)}
\\
\le N \|f\|_{L_p(\bR_T \times \bR^d_+)},
\end{multline*}
where $N=N(d,\delta,p)>0$.
Moreover, for any $f \in L_p(\bR_T \times \bR^d_+)$ and $\lambda > 0$, there exists a unique $u \in W_p^{1,2}(\bR_T \times \bR^d_+)$ satisfying \eqref{eq1203} with the Neumann boundary condition $D_1u=0$ on $\bR_T \times \partial \bR^d_+$.
\end{theorem}

\begin{remark}
Although in Theorems \ref{thm01} and \ref{thm02} we only consider equations without lower-order terms, it is worth noting that by following the proof of Corollary \ref{cor0411_1} the theorems can be extended to general linear equations with lower order terms:
$$
-u_t + a^{ij} D_{ij} u + b^i D_i u + c u - \lambda u = f
$$
as long as $b^i$ and $c$ are bounded measurable functions.
In particular, we obtain the a priori estimates in Theorems \ref{thm01} and \ref{thm02} for $\lambda \ge \lambda_0$, where $\lambda_0$ is determined by $d$, $\delta$, and the bound of the coefficients of $b^i$ and $c$.

If we consider equations as above, for example, in $[0,T] \times \Xi^k$ with appropriate initial conditions, then we are able to remove the term $\lambda u$ by considering $u e^{-\lambda t}$.
In this case the constant $N$ depends on $T$ as well. See, for instance, the proof of Theorem 2.1 in \cite{Kr07}.
\end{remark}

As an application of Theorem \ref{thm01}, by using a change of variables and the partition of unity argument, and following the steps in Chapter 11 of
\cite{Kr08}, we derive the corresponding $W^{1,2}_p$-estimate and solvability of parabolic equations of the form
$$
-u_t + a^{ij} D_{ij} u + b^i D_i u + c u - \lambda u = f
$$
in $\bR_T\times \Omega$, where $\Omega$ is a general simple convex polytope, $\lambda>0$ is a constant, and $a^{ij}=a^{ij}(t)$ are measurable functions in the time variable. If, in addition, $\Omega$ is bounded and $c\ge 0$, we can take $\lambda$ to be zero. See, for instance, Chapter 11 of \cite{Kr08}.

\section{Non-divergence type equations when $p=2$}
							\label{secNonDiv2}

In this section we prove Theorem \ref{thm01} when $p=2$ with more general boundary conditions.

\begin{theorem}
							\label{thm0411}
Let $T \in (-\infty,\infty]$, $\lambda \ge 0$, $k \in [1,d]$ be an integer, and $f\in L_2(\bR_T \times \Xi^k)$.
Assume that $a^{ij}$ satisfy \eqref{eq5.21} and \eqref{eq5.22}.
If $u \in W_2^{1,2}(\bR_T \times \Xi^k)$ satisfies
\begin{equation}
								\label{eq0410-1}
- u_t + a^{ij} D_{ij}u - \lambda u = f
\end{equation}
in $\bR_T \times \Xi^k$ with either the Dirichlet boundary condition $u = 0$ or the Neumann boundary condition $D_1 u=0$ on $\bR_T \times \Gamma^{k,1}$, and the Dirichlet boundary condition $u = 0$ on the other faces $\bR_T \times \Gamma^{k,i}$, $i=2,\ldots,k$ (when $k\ge 2$),
then we have
\begin{multline}
							\label{eq0410-2}
\lambda \|u\|_{L_2(\bR_T \times \Xi^k)} + \lambda^{1/2}\|Du\|_{L_2(\bR_T \times \Xi^k)} + \|D^2u\|_{L_2(\bR_T \times \Xi^k)}
+ \|u_t\|_{L_2(\bR_T \times \Xi^k)}
\\
\le N \|f\|_{L_2(\bR_T \times \Xi^k)},
\end{multline}
where $N=N(d,\delta)>0$.
Moreover, for any $f \in L_2(\bR_T \times \Xi^k)$ and $\lambda > 0$, there exists a unique $u \in W_2^{1,2}(\bR_T \times \Xi^k)$ satisfying \eqref{eq0410-1} with the boundary conditions described above.
\end{theorem}

It is an interesting question whether the estimate above still holds when $2\le k<d$ and $u$ satisfies the Neumann boundary condition on at least two faces.

\begin{proof}[Proof of Theorem \ref{thm0411}]
We first observe that the statements in the theorem hold true if $a^{ij} = \delta_{ij}$. This can be justified by extending the equation to the domain $\bR_T \times \bR^d$. A similar extension process is explained in the derivation of \eqref{eq0425_1} below.
Thus, thanks to the method of continuity, it is enough to prove the a priori estimate \eqref{eq0410-2}.  Without loss of generality, we assume that $u \in C^\infty(\overline{\bR_T \times \Xi^k}) \cap W_2^{1,2}(\bR_T \times \Xi^k)$ and $a^{ij}$ are infinitely differentiable with bounded derivatives. Moreover, because of Remark \ref{rem2.2}, we may assume that $k\le d-1$.

We first claim that
\begin{equation}
							\label{eq0411-2}
\sum_{\substack{i,j=1\\ (i,j) \ne (d,d)}}^d \|D_{ij} u\|_{L_2(\bR_T \times \Xi^k)} + \sum_{i=1}^{d-1} \sqrt{\lambda} \|D_i u\|_{L_2(\bR_T \times \Xi^k)} \le N \|f\|_{L_2(\bR_T \times \Xi^k)},
\end{equation}
where $N=N(d,\delta)$.
For the proof, we split into two cases: $a^{dd}= a^{dd}(t)$ and $a^{dd}=a^{dd}(x_d)$.

{\em (i) The case $a^{dd}= a^{dd}(t)$}:
We first consider the case when $u=0$ on $\bR_T\times \Gamma^{k,1}$.
Note that
\begin{align}
                                \label{eq2.38}
\int_{\bR_T \times \Xi^k} a^{dd} D^2_d u \, D_1^2 u \, dx \, dt
&= - \int_{\bR_T \times \Xi^k} a^{dd} D_1D^2_d u \, D_1 u \, dx \, dt\nonumber\\
&=  \int_{\bR_T \times \Xi^k} a^{dd} D_{1d} u \, D_{1d} u \, dx \, dt,
\end{align}
since $D_d^2 u=0$ on $\bR_T\times \Gamma^{k,1}$.
For $2\le i\le d-1$ and $2\le j\le d$,
\begin{align}
                                \label{eq2.38b}
\int_{\bR_T \times \Xi^k} a^{ij} D_{ij} u \, D_1^2 u \, dx \, dt
&= - \int_{\bR_T \times \Xi^k} a^{ij} D_{1ij} u \, D_1 u \, dx \, dt\nonumber\\
&=  \int_{\bR_T \times \Xi^k} a^{ij} D_{1j} u \, D_{1i} u \, dx \, dt,
\end{align}
where in the first equality we used $D_{ij}u=0$ on $\bR_T\times \Gamma^{k,1}$, and in the second equality we integrated by parts in $x_i$ and used $D_1 u=0$ on $\bR_T\times \Gamma^{k,i}$. Of course, \eqref{eq2.38b} still holds if $2\le i\le d$ and $2\le j\le d-1$, in which case we integrate by parts in $x_j$ in the second equality.
If at least one of $i$ and $j$ equals $1$, clearly we still have
\begin{equation}
                                        \label{eq3.04}
\int_{\bR_T \times \Xi^k} a^{ij} D_{ij} u \, D_1^2 u \, dx \, dt=\int_{\bR_T \times \Xi^k} a^{ij} D_{1j} u \, D_{1i} u \, dx \, dt
\end{equation}
without performing any integration by parts.
Thus by multiplying both sides of the equation \eqref{eq0410-1} by $D_1^2 u$, and integrating by parts, whenever necessary, we see that
\begin{align*}
&\int_{\bR_T \times \Xi^k} - u_t \, D_1^2 u \, dx \, dt
+ \int_{\bR_T \times \Xi^k} a^{ij} D_{ij} u \, D_1^2 u \, dx \, dt
- \lambda \int_{\bR_T \times \Xi^k} u \, D_1^2 u \, dx \, dt\\
&= \frac 1 2 \int_{\bR_T \times \Xi^k} \frac{\partial}{\partial t} \left( D_1 u \right)^2 \, dx \, dt
+ \int_{\bR_T \times \Xi^k} a^{ij} D_{1i} u \, D_{1j} u \, dx \, dt
+ \lambda \int_{\bR_T \times \Xi^k} |D_1 u|^2 \, dx \, dt\\
&= \int_{\bR_T \times \Xi^k} f D_1^2 u \, dx \, dt,
\end{align*}
which together with \eqref{eq5.21} and H\"{o}lder's inequality implies that
\begin{equation}
							\label{eq0411-1}
\|DD_1 u\|_{L_2(\bR_T \times \Xi^k)} + \sqrt{\lambda} \|D_1u\|_{L_2(\bR_T \times \Xi^k)} \le N(d,\delta) \|f\|_{L_2(\bR_T \times \Xi^k)}.
\end{equation}
By symmetry, for $i = 2, \ldots, k$, we have \eqref{eq0411-1} with $DD_i u$
and $D_iu$ in places of $DD_1u$ and $D_1u$, respectively. The same estimate holds for $i=k+1,\ldots,d-1$ by similar and actually simpler reasoning.
Therefore, we arrive at \eqref{eq0411-2}.

Now we treat the case when $u$ satisfies the Neumann boundary condition on $\bR_T\times \Gamma^{k,1}$. Since $D_1u=0$ on $\bR_T \times \Gamma^{k,1}$ and $\bR_T\times \Gamma^{k,i},i=2,\ldots,k$, using integration by parts as above, we obtain the equalities \eqref{eq2.38}, \eqref{eq2.38b}, and \eqref{eq3.04}, which imply \eqref{eq0411-1}. We then use the idea in Remark \ref{rem2.2} to reduce the number of the faces by $1$. Rewrite the equation \eqref{eq0410-1} as
$$
-u_t + a^{11}D_{11}u+\sum_{i,j=2}^d a^{ij} D_{ij} u - \lambda u = f - \sum_{i=2}^d (a^{i1}+a^{1i}) D_{1i}u :=F.
$$
Let $\bar{u}$ and $\bar{F}$ be the even extensions of $u$ and $F$ with respect to $x_1$ to the domain $\bR_T \times \Omega$, where
\begin{equation}
							\label{eq0424_3}
\Omega:= \bR \times \Xi^{k-1}_{d-1}:= \{ x \in \bR^d : x_2, \ldots, x_k \in \bR^+\}.
\end{equation}
It is clear that the coefficients $a^{11}$ and $a^{ij}$, $i,j \ge 2$, satisfy \eqref{eq5.21} and \eqref{eq5.22}.
Moreover, $\bar{u} \in W_2^{1,2}(\bR_T \times \Omega)$ satisfies
\begin{equation*}
-\bar{u}_t +    a^{11}D_{11} \bar u+\sum_{i,j=2}^d a^{ij} D_{ij} \bar u - \lambda \bar{u} = \bar{F}
\end{equation*}
in $\bR_T \times \Omega$ with the Dirichlet boundary condition $u = 0$ on $\bR_T \times \Gamma^{k,i}$, $i=2,\ldots,k$ (when $k\ge 2$). Rearrange the coordinates
$$
(x_1, x_2, \ldots, x_d) \to (x_2, \ldots, x_k, x_1, x_{k+1}, \ldots, x_d).
$$
Then
from the estimate \eqref{eq0411-2} for the Dirichlet boundary condition case above, we have
$$
\sum_{\substack{i,j=1\\ (i,j) \ne (d,d)}}^d \|D_{ij} \bar u\|_{L_2(\bR_T \times \Omega)} + \sum_{i=1}^{d-1} \sqrt{\lambda} \|D_i u\|_{L_2(\bR_T \times \Omega)} \le N \|F\|_{L_2(\bR_T \times \Omega)},
$$
which together with the definition of $\bar u$ and $\bar F$ as well as \eqref{eq0411-1} implies \eqref{eq0411-2} in this case.

{\em (ii) The case $a^{dd}=a^{dd}(x_d)$}: We make a change of variables to exploit the divergence structure of the equation. Let
$$
y_d := \phi(x_d) = \int_0^{x_d} \frac{1}{a^{dd}(s)} \, ds,
\quad
y_j := x_j,
\quad
j =1, \ldots, d-1.
$$
Denote $y' = (y_1,\ldots, y_{d-1})$ and set
$$
\tilde u(t,y) = u\left(t,y', \phi^{-1}(y_d)\right),
\quad
\tilde f(t,y) = f\left(t,y', \phi^{-1}(y_d)\right),
$$
$$
\tilde{a}^{dd}(y_d) = 1/a^{dd}(\phi^{-1}(y_d)),
$$
$$
\tilde{a}^{jd}(t,y_d) = \frac{a^{d j}+a^{j d}}{a^{dd}}\left(t,\phi^{-1}(y_d)\right),\quad \tilde{a}^{dj} = 0,
\quad j =1, \ldots, d-1,
$$
and
$\tilde{a}^{ij}(t,y_d) = a^{ij}(t,\phi^{-1}(y_d))$ for the other $(i,j)$.
Then we see that $\tilde u$ satisfies
$$
-\tilde{u}_t + \sum_{\substack{i,j=1\\ (i,j) \ne (d,d)}}^d \tilde{a}^{ij} D_{ij} \tilde u
+ D_d \left[ \tilde{a}^{dd} D_d \tilde{u} \right]
- \lambda \tilde{u} = \tilde{f}
$$
in $\bR_T \times \Xi^k$ with the same boundary conditions as $u$.
Note that
\begin{align*}
\int_{\bR_T \times \Xi^k} D_d \left[ \tilde{a}^{dd} D_d \tilde{u} \right] D_1^2 \tilde{u} \, dx \, dt
&= - \int_{\bR_T \times \Xi^k} D_d \left[ \tilde{a}^{dd} D_{1d} \tilde{u} \right] D_1 \tilde{u} \, dx \, dt\\
&=  \int_{\bR_T \times \Xi^k} \tilde{a}^{dd} D_{1d} \tilde{u} \, D_{1d} \tilde{u} \, dx \, dt
\end{align*}
if $u$ satisfies either $u=0$ or $D_1 u=0$ on $\bR_T\times\Gamma^{k,1}$.
Since the coefficients $\tilde{a}^{ij}$ satisfy \eqref{eq5.21} with an ellipticity constant depending only on $\delta$, by proceeding as above, we again obtain \eqref{eq0411-2}.

Now we prove \eqref{eq0410-2}. We again use the extension method. First we write the equation \eqref{eq0410-1} as
$$
-u_t + \sum_{i=1}^d a^{ii} D_{ii} u - \lambda u = f - \sum_{\substack{i,j=1\\ i \ne j}}^d a^{ij} D_{ij}u :=F.
$$
Set $\bar{u}$ and $\bar{F}$ to be the odd extensions of $u$ and $F$ with respect to $x_k$ if $u=0$ on $\bR_T\times \Gamma^{k,k}$.
In case $k = 1$ and $D_1 u = 0$ on $\bR_T \times \Gamma^{1,1}$, we set $\bar{u}$ and $\bar{F}$ to be the even extensions of $u$ and $F$ with respect to $x_1$. Then we have
$$
\|\bar{u}\|_{W_2^{1,2}(\bR_T \times \Xi^{k-1})} \cong \|u\|_{W_2^{1,2}(\bR_T \times \Xi^k)},
\quad
\|\bar F\|_{L_2(\bR_T \times \Xi^{k-1})} \cong \|F\|_{L_2(\bR_T \times \Xi^k)},
$$
and $\bar{u} \in W_2^{1,2}(\bR_T\times\Xi^{k-1})$ satisfies
\begin{equation}
							\label{eq0425_1}
-\bar{u}_t + \sum_{i=1}^d a^{ii} D_{ii} \bar{u} - \lambda \bar{u} = \bar{F}
\end{equation}
in $\bR_T\times \Xi^{k-1}$ with the same boundary conditions as $u$ on $\bR_T \times \Gamma^{k,i}$, $i=1, \ldots, k-1$.
We do not need any extensions of $a^{ij}$ because $a^{ij}$ are functions of only $t$ and $x_d$.
We repeat this extension process described above to consecutively extend an equation in $\bR_T\times \Xi^{k-i}$ to an equation in $\bR_T\times\Xi^{k-i-1}$, $i=0,\ldots,k-1$, until we get the equation \eqref{eq0425_1} in $\bR_T \times \bR^d$.
Now we observe that $a^{ii}$ are measurable functions of time and one spatial variable.
Then by the $L_2$-estimate for equations in the whole space obtained in \cite[Theorem 3.2]{KK07} together with the definition of $\bar{F}$, we obtain
\begin{multline*}
\lambda \|\bar{u}\|_{L_2(\bR_T \times \bR^d)} + \lambda^{1/2}\|D\bar{u}\|_{L_2(\bR_T \times \bR^d)} + \|D^2\bar{u}\|_{L_2(\bR_T \times\bR^d)}
+ \|\bar{u}_t\|_{L_2(\bR_T \times \bR^d)}\\
\le N(d,\delta) \|\bar{F}\|_{L_2(\bR_T \times \bR^d)}
\le N \sum_{i \ne j} \|D_{ij} u\|_{L_2(\bR_T \times \Xi^k)} + N \|f\|_{L_2(\bR_T \times \Xi^k)}.
\end{multline*}
Finally, we use the estimate \eqref{eq0411-2} and the comparability of the $L_2$ norms of functions in $\bR_T \times \Xi^k$ and those of their extensions in $\bR_T \times \bR^d$. The theorem is proved.
\end{proof}

\begin{corollary}
							\label{cor0411_1}
Let $T \in (-\infty,\infty]$, $k \in [1,d]$ be an integer, $b^i$ and $c$ be measurable functions bounded by a positive constant $K$, and $f \in L_2(\bR_T \times \Xi^k)$.
Assume that $a^{ij}$ satisfy \eqref{eq5.21} and \eqref{eq5.22}. Then there exists $\lambda_0=\lambda_0(d,\delta,K) \ge 0$ such that, for any $\lambda \ge \lambda_0$ and $u \in W_2^{1,2}(\bR_T \times \Xi^k)$ satisfying
\begin{equation}
							\label{eq0424_06}
-u_t + a^{ij}D_{ij}u + b^i D_i u + c u - \lambda u  = f
\end{equation}
in $\bR_T \times \Xi^k$ with the boundary conditions on $\bR_T \times \partial \Xi^k$ as stated in Theorem \ref{thm0411}, we have the a priori estimate \eqref{eq0410-2}.
Moreover, for any $f \in L_2(\bR_T \times \Xi^k)$ and $\lambda > \lambda_0$, there exists a unique $u \in W_2^{1,2}(\bR_T \times \Xi^k)$ satisfying \eqref{eq0424_06} with the boundary conditions as in Theorem \ref{thm0411}.
\end{corollary}

\begin{proof}
Thanks to Theorem \ref{thm0411} and the method of continuity, it suffices to show \eqref{eq0410-2} for sufficiently large $\lambda$. To this end, we rewrite \eqref{eq0424_06} as
$$
-u_t + a^{ij}D_{ij}u  - \lambda u  = f-b^i D_i u - c u.
$$
By Theorem \ref{thm0411}, we have
\begin{multline*}
\lambda \|u\|_{L_2(\bR_T \times \Xi^k)} + \lambda^{1/2}\|Du\|_{L_2(\bR_T \times \Xi^k)} + \|D^2u\|_{L_2(\bR_T \times \Xi^k)}
+ \|u_t\|_{L_2(\bR_T \times \Xi^k)}
\\
\le N \|f\|_{L_2(\bR_T \times \Xi^k)}+N\|Du\|_{L_2(\bR_T \times \Xi^k)}+N\|u\|_{L_2(\bR_T \times \Xi^k)},
\end{multline*}
where the constant $N=N(d,\delta,K)$ is independent of $\lambda$.
This implies \eqref{eq0410-2} for $\lambda\ge \lambda_0$, where $\lambda_0=\lambda_0(d,\delta,K) \ge 0$ is sufficiently large. The corollary is proved.
\end{proof}

\section{Parabolic equations in divergence form}
                                        \label{sec3}

In this section, we consider divergence form parabolic operators of two special types.
Throughout the section, we set
\begin{equation}
							\label{eq1206}
\cL u = D_i (a^{ij} D_j u),
\end{equation}
where $a^{ij}$ satisfy \eqref{eq5.21} and \eqref{eq5.22}.
We denote $\cH_p^1$ to be the solution spaces for divergence form parabolic equations. Precisely,
$$
\cH^1_p((S,T) \times \Omega)=
\{u: u,Du \in L_p((S,T) \times \Omega),u_t \in \bH^{-1}_p((S,T) \times \Omega)\},
$$
where $\bH^{-1}_p((S,T) \times \Omega)$ is the space consisting of all generalized functions $v$ satisfying
$$
\inf \big\{\|f\|_{L_p((S,T) \times \Omega)}+\|g\|_{L_p((S,T) \times \Omega)}\,|\,v=\Div f+g\big\}<\infty.
$$

\subsection{Auxiliary results}
The following theorem is a simple consequence of the Lax--Milgram lemma.

\begin{theorem}
							\label{thm0411_1}
Let $\lambda > 0$, $T \in (-\infty,\infty]$, $k \in [1,d]$ be an integer, and $f = (f_1,\ldots,f_d), g \in L_2(\bR_T \times \Xi^k)$. Then there exists a unique $u \in \cH_2^1(\bR_T \times \Xi^k)$ satisfying
$$
-u_t + \cL u - \lambda u = \Div f + g
$$
in $\bR_T \times \Xi^k$ with either the Dirichlet boundary condition $u = 0$ or the conormal derivative boundary condition $a^{ij}D_j u=f_i$  on each $\bR_T \times \Gamma^{k,i}$, where $i=1,\ldots,k$.
Furthermore, we have
$$
\sqrt\lambda\|u\|_{L_2(\bR_T\times  \Xi^k)}+\|D u\|_{L_2(\bR_T \times  \Xi^k)}
\le N \|f\|_{L_2(\bR_T \times \Xi^k)} + N \lambda^{-1/2} \|g\|_{L_2(\bR_T \times \Xi^k)},
$$
where $N=N(d,\delta)$.
\end{theorem}

Throughout the rest of the paper, by $\bL_1$ we mean the collection of divergence form operators as in \eqref{eq1206} with $a^{i1} = 0$ for $i = 2, \ldots, d$. Similarly, $\bL_2$ is the collection of divergence form operators satisfying $a^{1j} = 0$ for $j=2,\ldots,d$.

When the operator $\cL$ is either in $\bL_1$ or $\bL_2$, we obtain the following observations in Lemmas \ref{lem02} and \ref{lem01} below.

\begin{lemma}
							\label{lem02}
Let $\cL \in \bL_2$, $\lambda \ge 0$,  $k \in [1,d]$ be an integer, $R \in (0,\infty]$,
and $f \in L_2(Q_R^k)$.
If $u \in \cH^1_2(Q_R^k)$ satisfies
\begin{equation*}
- u_t + \cL u - \lambda u = f
\end{equation*}
in $Q_R^k$ with the conormal derivative boundary condition $a^{11} D_1 u = 0$ on $(-R^2,0)\times \Gamma_R^{k,1}$ and the Dirichlet boundary condition $u = 0$ on $(-R^2,0)\times \Gamma_R^{k,i}$, $i = 2,\ldots,k$, then $w: = D_1u$ belongs to $\cH_2^1(Q_r^k)$ for any $r\in (0,R)$ and satisfies
\begin{equation}
							\label{eq1211}
-w_t + \cL w -\lambda w = D_1 f
\end{equation}
in $Q_r^k$
with the Dirichlet boundary condition $w = 0$
on $(-r^2,0) \times \left(B_r \cap \partial \Xi^k\right)$.
\end{lemma}
\begin{proof}
Take an infinitely differentiable function $\varphi(t,x)$ defined in $\bR \times \bR^d$ with a compact support in $(-R^2,R^2) \times B_R$ such that $\varphi$ is even with respect to $x_1$.
Note that $D_1 \varphi = 0$ on $\bR \times \Gamma^{k,1}$.
Set
\begin{equation}
							\label{eq0604_3}
v := \varphi u,
\quad
h_i := a^{ij} u D_j \varphi,
\quad
g := a^{ij} (D_i \varphi) (D_j u) - \varphi_t u + \varphi f.
\end{equation}
Note that $h_i,g\in L_2(\bR_0\times \Xi^k)$.
Then $v$ satisfies
\begin{equation}
							\label{eq0604_4}
- v_t + \cL v - \lambda v = D_i h_i + g
\end{equation}
in $\bR_0 \times \Xi^k$ with the conormal derivative boundary condition
$$
a^{11}D_1 v = a^{11} u D_1 \varphi = 0=h_1
$$
on $\bR_0 \times \Gamma^{k,1}$ and the Dirichlet boundary condition $v = 0$ on $\bR_0 \times \Gamma^{k,i}$, $i=2,\ldots,k$.
For $\varepsilon>0$, let $a^{ij}_\varepsilon$ be the standard mollifications of $a^{ij}$ with respect to $(t,x_d)$ and $h_i^\varepsilon := a^{ij}_\varepsilon u D_j \varphi$.
Note that, due to the choice of $\varphi$, we have $h_1^\varepsilon = 0$ on $\bR_0 \times \Gamma^{k,1}$.
Using Theorem \ref{thm0411_1} we find a unique solution $v^\varepsilon \in \cH_2^1(\bR_0 \times \Xi^k)$ to the equation
\begin{equation}
								\label{eq0603_1}
- v^\varepsilon_t + D_i\left(a^{ij}_\varepsilon D_j v^\varepsilon\right) - \lambda v^\varepsilon = D_i h_i^\varepsilon + g
\end{equation}
in $\bR_0 \times \Xi^k$ with the conormal derivative boundary condition
$$
a^{11}_\varepsilon D_1 v^\varepsilon = h_1^\varepsilon = 0
$$
on $\bR_0 \times \Gamma^{k,1}$ and the Dirichlet boundary condition $v^\varepsilon = 0$ on $\bR_0 \times \Gamma^{k,i}$, $i=2,\ldots,k$. Since $h_i^\varepsilon\to h_i$ in $L_2(\bR_0\times \Xi^k)$, we see that $v^\varepsilon \to v$ in $\cH_2^1(\bR_0 \times \Xi^k)$. In particular,
\begin{equation}
							\label{eq0604_1}
D_1 v^{\varepsilon} \to D_1 v
\quad
\text{in}
\,\,
L_2(\bR_0 \times \Xi^k).
\end{equation}

Take $\lambda_{0,\varepsilon}$ to be a number bigger than $\lambda_0$  in Corollary \ref{cor0411_1} determined by $d$, $\delta$, and the bound of $D_i(a^{ij}_\varepsilon)$.
Since $D_i h_i^\varepsilon + g \in L_2(\bR_0 \times \Xi^k)$, by Corollary \ref{cor0411_1} there exists a unique $u^\varepsilon \in W_2^{1,2}(\bR_0 \times \Xi^k)$ satisfying
\begin{equation*}
- u^\varepsilon_t + a^{ij}_\varepsilon D_{ij} u^\varepsilon + D_i(a^{ij}_\varepsilon) D_j u^\varepsilon - \lambda_{0,\varepsilon} u^\varepsilon = D_i h_i^\varepsilon + g + (\lambda - \lambda_{0,\varepsilon}) v^\varepsilon
\end{equation*}
in $\bR_0 \times \Xi^k$ with the Neumann boundary condition $D_1u^\varepsilon = 0$ on $\bR_0 \times \Gamma^{k,1}$ and the Dirichlet boundary condition $u^\varepsilon = 0$ on $\bR_0 \times \Gamma_{k,i}$, $i=2,\ldots,k$.
Thanks to the boundary conditions of $u^\varepsilon$ and the facts that $\cL \in \bL_2$ and  $h_1^\varepsilon = 0$ on $\bR_0 \times \Gamma^{k,1}$, $u^\varepsilon$ satisfies the divergence type equation
$$
- u^\varepsilon_t + D_i (a^{ij}_\varepsilon D_j u^\varepsilon) - \lambda_{0,\varepsilon} u^\varepsilon = D_i h_i^\varepsilon + g + (\lambda - \lambda_{0,\varepsilon}) v^\varepsilon
$$
in $\bR_0 \times \Xi^k$ with the same boundary conditions for \eqref{eq0603_1}.
Since $u^\varepsilon \in \cH_2^1(\bR_0 \times \Xi^k)$ and $v^\varepsilon$ also satisfies the above equation with the same boundary conditions, by the uniqueness, we have $v^\varepsilon = u^\varepsilon \in W_2^{1,2}(\bR_0 \times \Xi^k)$, $D_1v^\varepsilon = 0$ on $\bR_0 \times \Gamma^{k,1}$, and  $v^\varepsilon = 0$ on $\bR_0 \times \Gamma_{k,i}$, $i=2,\ldots,k$.
Then from \eqref{eq0603_1}, we see that $w^\varepsilon := D_1 v^\varepsilon \in \cH_2^1(\bR_0 \times \Xi^k)$ satisfies the divergence type equation
\begin{equation*}
- w^\varepsilon_t + D_i\left(a^{ij}_\varepsilon D_j w^\varepsilon\right) - \lambda w^\varepsilon = D_1(D_1 h_1^\varepsilon + g) + \sum_{i=2}^d D_i (D_1 h_i^\varepsilon)
\end{equation*}
in $\bR_0\times\Xi^k$ with the Dirichlet boundary condition $w^\varepsilon = 0$ on $\bR_0 \times \partial \Xi^k$.
Since $a^{ij}$ and $a^{ij}_\varepsilon$ are independent of $x_1$, we have
$$
D_1 h_i = a^{ij} D_1(uD_j \varphi),
\quad
D_1 h_i^\varepsilon = a^{ij}_\varepsilon D_1(uD_j \varphi),
$$
and $D_1 h_i^\varepsilon \to D_1 h_i$ in $L_2(\bR_0 \times \Xi^k)$.
Thus $w^\varepsilon$ converges in $\cH^1_2(\bR_0\times \Xi^k)$ to the unique solution $\tilde{w} \in \cH_2^1(\bR_0 \times \Xi^k)$ to the equation
\begin{equation}
                                    \label{eq11.32}
-\tilde{w}_t + D_i (a^{ij} D_j \tilde{w} ) - \lambda \tilde{w} = D_1( D_1 h + g) + \sum_{i=2}^d D_i (D_1h_i)
\end{equation}
in $\bR_0 \times \Xi^k$ with the Dirichlet boundary condition $\tilde{w} = 0$ on $\bR_0 \times \partial \Xi^k$.
Especially,
$$
D_1 v^\varepsilon = w^\varepsilon \to \tilde{w}
\quad
\text{in}
\,\,
L_2(\bR_0 \times \Xi^k).
$$
This combined with \eqref{eq0604_1} proves that $D_1(\varphi u)=D_1v=\tilde{w}\in \cH^1_2(\bR_0\times\Xi^k)$.
Therefore, upon choosing an appropriate cut-off function $\varphi$, we conclude that $w=D_1 u$ satisfies \eqref{eq1211} with the desired boundary conditions.
\end{proof}

\begin{lemma}
							\label{lem01}
Let $\cL \in \bL_1$, $\lambda \ge 0$, $k \in [1,d]$ be an integer, $R \in (0,\infty]$,
and $f \in L_2(Q_R^k)$.
If $u \in \cH^1_2(Q_R^k)$ satisfies
\begin{equation*}
- u_t + \cL u -\lambda u = f
\end{equation*}
in $Q_R^k$ with the Dirichlet boundary condition $u = 0$ on $(-R^2,0)\times \left(B_R \cap \partial \Xi^k\right)$,
then we have $w: = D_1u$ belongs to $\cH^1_{2}(Q_r^k)$ for any $r\in (0,R)$ and satisfies
\begin{equation}
							\label{eq0604_5}
-w_t + \cL w -\lambda w = D_1 f
\end{equation}
in $Q_r^k$
with the conormal derivative boundary condition
$a^{1j} D_j w = f$ on $(-r^2,0) \times \Gamma_r^{k,1}$
and the Dirichlet boundary condition $w=0$ on $(-r^2,0) \times \Gamma_r^{k,i}$, $i=2, \ldots,k$.
\end{lemma}

\begin{proof}
Take $\varphi(t,x)$ as in the proof of Lemma \ref{lem02} (not necessarily to be even in $x_1$ this time).
Set $v$, $h_i$, and $g$ as in \eqref{eq0604_3} and define $a^{ij}_\varepsilon$ and $h_i^\varepsilon$ as in the proof of Lemma \ref{lem02}. Then $v$ satisfies \eqref{eq0604_4}
in $\bR_0 \times \Xi^k$ with the Dirichlet boundary condition $v = 0$ on $\bR_0 \times \partial \Xi^k$. Let $v^\varepsilon \in \cH_2^1(\bR_0 \times \Xi^k)$ be the unique solution of  \eqref{eq0603_1} in $\bR_0\times \Xi^k$ with the Dirichlet boundary condition $v^\varepsilon = 0$ on $\bR_0\times \partial \Xi^k$. By the same argument, we have \eqref{eq0604_1} and $v^\varepsilon\in W^{1,2}_2(\bR_0\times \Xi^k)$.

Here it is more involved to verify the conormal derivative boundary condition since it is understood in the weak sense. To this end, we use an extension argument. Set $\bar{v}^\varepsilon$ to be the odd extension of $v^\varepsilon$ with respect to $x_1$. Since $v^\varepsilon \in W_2^{1,2}(\bR_0 \times \Xi^k)$ with the Dirichlet boundary condition on $\bR_0 \times \Gamma^{k,1}$, it follows that $\bar{v}^\varepsilon \in W_2^{1,2}(\bR_0 \times \Omega)$, where $\Omega$ is defined in \eqref{eq0424_3}.
Set $\bar{a}^{1j}_\varepsilon$, $j=2,\ldots, d$, to be the odd extensions of $a^{1j}_\varepsilon$ in $x_1$.
For the other $(i,j)$, we set $\bar{a}^{ij}_\varepsilon = a^{ij}_\varepsilon$.
Also set $\bar{h}_1^\varepsilon$ to be the even extension of $h_1^\varepsilon$, $\bar{h}_i^\varepsilon$, $i=2, \ldots,d$, to be the odd extensions of $h_i^\varepsilon$, and $\bar{g}$ to be the odd extension of $g$.
Then $\bar{v}^\varepsilon$ satisfies
$$
- \bar{v}^\varepsilon_t + D_i(\bar{a}^{ij}_\varepsilon D_j \bar{v}^\varepsilon) - \lambda \bar{v}^\varepsilon =  D_i \bar{h}_i^\varepsilon + \bar{g}
$$
in $\bR_0 \times \Omega$ with the Dirichlet boundary condition on $\bR_0 \times \partial \Omega$.

Since $u = 0$ on $\bR_0 \times \Gamma^{k,1}$, we see that $D_1 \bar{h}_1^\varepsilon \in L_2(\bR_0 \times \Omega)$ and $D_1 \bar{h}_1^\varepsilon$ is the odd extension of $D_1 h_1^\varepsilon$ in $x_1$, and $D_1 \bar{h}_i^\varepsilon \in L_2(\bR_0 \times \Omega)$, $i=2,\ldots,d$, and $D_1 \bar{h}^\varepsilon_i$ are the even extensions of $D_1 h_i^\varepsilon$ in $x_1$.
From the fact that $v^{\varepsilon} = 0$ on $\bR_0 \times \partial \Xi^k$, $\cL\in \bL_1$, and integrating by parts in $x_1$, it follows that $w^{\varepsilon}:=D_1 \bar{v}^\varepsilon \in \cH_2^1(\bR_0 \times \Omega)$ satisfies
\begin{equation}
							\label{eq0604_2}
- w^{\varepsilon}_t + D_i\left(\bar{a}^{ij}_\varepsilon D_j w^{\varepsilon}\right) - \lambda w^{\varepsilon} =  D_1 \left( D_1 \bar{h}_1^\varepsilon + \bar{g} \right) + \sum_{i=2}^d D_i ( D_1 \bar{h}_i^\varepsilon)
\end{equation}
in $\bR_0 \times \Omega$ with the Dirichlet boundary condition on $\bR_0 \times \partial \Omega$.
In fact, $\bar a^{ij}_\varepsilon$ are independent of $x_1$, except $\bar a^{1j}_\varepsilon$, $j=2,\ldots,d$, which have jump discontinuities at $x_1=0$. But for these $j$ and any $\phi\in C_0^\infty(\bR\times\Omega)$,
\begin{align*}
&\int_{\bR_0\times\Omega}\bar a^{1j}_\varepsilon D_j w^\varepsilon D_1\phi\,dx\,dt\\
&=\int_{\bR_0\times\Xi^k}\bar a^{1j}_\varepsilon D_j w^\varepsilon D_1\phi\,dx\,dt+\int_{\bR_0\times(\Omega\setminus\Xi^k)}\bar a^{1j}_\varepsilon D_j w^\varepsilon D_1\phi\,dx\,dt\\
&=-\int_{\bR_0\times\Xi^k}\bar a^{1j}_\varepsilon D_j \bar v^\varepsilon D_1^2\phi\,dx\,dt-\int_{\bR_0\times(\Omega\setminus\Xi^k)}\bar a^{1j}_\varepsilon D_j \bar v^\varepsilon D_1^2\phi\,dx\,dt\\
&\qquad -2\int_{\bR_0\times\Gamma^{k,1}} \bar a^{1j}_\varepsilon D_j \bar v^\varepsilon D_1\phi\,dx\,dt\\
&=-\int_{\bR_0\times\Omega}\bar a^{1j}_\varepsilon D_j \bar v^\varepsilon D_1^2\phi\,dx\,dt,
\end{align*}
where in the last equality we have used $D_j \bar v^\varepsilon=0$ on $\bR_0\times\Gamma^{k,1}$. We note that \eqref{eq0604_2} is the only place in the proof where we use $\cL\in \bL_1$, i.e., $ a^{i1}=0$, $i=2,\ldots,d$.

On the other hand, there is a unique solution $\tilde w^\varepsilon \in \cH_2^1(\bR_0 \times \Xi^k)$ to the equation
\begin{equation}
							\label{eq0704}
-\tilde w_t^\varepsilon + D_i(a^{ij}_\varepsilon D_j \tilde w^\varepsilon) - \lambda \tilde{w}^\varepsilon = D_1 \left(D_1 h_1^\varepsilon + g\right) + \sum_{i=2}^d D_i \left(D_1h_i^\varepsilon\right)
\end{equation}
with the conormal derivative boundary condition $a^{1j}_\varepsilon D_j \tilde w^\varepsilon = D_1 h_1^\varepsilon + g$ on $\bR_0 \times \Gamma^{k,1}$
and the Dirichlet boundary condition $\tilde w^\varepsilon = 0$ on $\bR_0 \times \Gamma^{k,i}$, $i=2,\ldots,k$.
Since $D_1 \bar{h}_1^\varepsilon + \bar{g}$ is the odd extension of $D_1 h_1^\varepsilon + g$ and
$D_1 \bar{h}_i^\varepsilon$, $i=2,\ldots,d$, are the even extensions of $D_1 h_i^\varepsilon$ with respect to $x_1$, the function $\bar{\tilde{w}}^\varepsilon \in \cH_2^1(\bR_0 \times \Omega)$, which is the even extension of $\tilde{w}^\varepsilon$ with respect to $x_1$, satisfies \eqref{eq0604_2} with the Dirichlet boundary condition on $\bR_0 \times \partial \Omega$.
By the uniqueness, we have $w^\varepsilon = \bar{\tilde{w}}^\varepsilon$.
This indicates that $\tilde{w}^\varepsilon$ in \eqref{eq0704} can be replaced by $w^\varepsilon$.
That is, as a function in $\cH_2^1(\bR_0 \times \Xi^k)$, $w^\varepsilon$ satisfies \eqref{eq0704} with the conormal derivative boundary condition $a^{1j}_\varepsilon D_j w^\varepsilon = D_1 h_1^\varepsilon + g$ on $\bR_0 \times \Gamma^{k,1}$ and the Dirichlet boundary condition $w^\varepsilon =0$ on $\bR_0 \times \Gamma^{k,i}$, $i= 2,\ldots,k$.
Then we see that, as in the proof of Lemma \ref{lem02}, $w^\varepsilon$ converges in $\cH^1_2(\bR_0\times \Xi^k)$ to the unique solution $\tilde{w} \in \cH_2^1(\bR_0 \times \Xi^k)$ of \eqref{eq11.32} with the conormal derivative boundary condition $a^{1j} D_j \tilde w= D_1 h_1+ g$ on $\bR_0 \times \Gamma^{k,1}$
and the Dirichlet boundary condition $\tilde w = 0$ on $\bR_0 \times \Gamma^{k,i}$, $i=2,\ldots,k$. In particular,
$$
D_1 v^\varepsilon = w^\varepsilon \to \tilde{w}
\quad
\text{in}
\,\,
L_2(\bR_0 \times \Xi^k).
$$
This combined with \eqref{eq0604_1} proves that $D_1(\varphi u)=D_1v=\tilde{w}\in \cH^1_2(\bR_0\times\Xi^k)$. Upon choosing an appropriate $\varphi$, we conclude that $w = D_1 u$ satisfies \eqref{eq0604_5} with the desired boundary conditions.
\end{proof}

\subsection{$C^{1,\alpha}$-estimates}

For functions defined on a subset $\cD$ in $\bR^{d+1}$, we denote
$$
[f]_{C^\alpha(\cD)}
= \sup_{\substack{(t,x), (s,y) \in \cD\\ (t,x) \ne (s,y)}} \frac{|f(t,x) - f(s,y)|}{|t-s|^{\alpha/2}+|x-y|^\alpha}.
$$

Notice that in the following lemma the operator $\cL$ is not necessarily in $\bL_1$ or $\bL_2$.

\begin{lemma}
                                \label{lem44}
Let $\lambda \ge 0$, $k\in [1,d]$ be an integer, and $u \in \cH^1_2(Q_2^k)$.
Suppose that $u$ satisfies
\begin{equation}
                                    \label{eq9.23}
-u_t+\cL u -\lambda u =0
\end{equation}
in $Q_2^k$ with the Dirichlet boundary condition $u = 0$ on $(-4,0)\times (B_2\cap \partial\Xi^k)$.
Then for $i \in \{ k+1, \ldots, d-1\}$, whenever this set is non-empty,
we have
\begin{equation}
                            \label{eq9.11}
[D_iu]_{C^\alpha(Q_1^{k})} \le N \|D_iu\|_{L_2(Q_2^{k})},
\end{equation}
where $\alpha = \alpha(d,\delta) \in (0,1)$
and $N = N(d,\delta)>0$ are constants.
If, in addition, $a^{dj}=0$ for $j=1,2,\ldots,l$, where $l = \min\{k, d-1\}$, then \eqref{eq9.11} holds true for $i = 1, \ldots, l$.
\end{lemma}

\begin{proof}
We first prove the case $\lambda = 0$.
If $i = k+1, \ldots, d-1$, then by using the standard difference quotient argument (see, for instance, Theorems 8.8 and 8.12 of \cite{GT}), we see that $w := D_i u$ is in  $\cH_{2}^1(Q_{3/2}^k)$ and satisfies
$$
-w_t + \cL w = 0
$$
in $Q_{3/2}^k$ with the Dirichlet boundary condition $w = 0$ on $(-4,0) \times \left(B_{3/2} \cap \partial \Xi^k \right)$.
Then using the De Giorgi--Nash--Moser estimate, we get
$$
[w]_{C^\alpha(Q_1^k)} \le N \|w\|_{L_2(Q_{3/2}^k)}
$$
for some $\alpha = \alpha(d,\delta) \in (0,1)$ and $N = N(d,\delta)$. Upon recalling that $w = D_i u$, we obtain \eqref{eq9.11}.

Now we assume that $1 \le i \le l$ and $a^{dj}=0$ for $j=1,\ldots, l$, where $l = \min\{k, d-1\}$. Due to the possibility of reordering the coordinates $(x_1,\cdots,x_l)$, it suffices to show
\begin{equation}
                                            \label{eq9.15}
[D_1u]_{C^\alpha(Q_1^{k})} \le N \|D_1u\|_{L_2(Q_2^{k})}.
\end{equation}
Let $\tilde a^{j1}=0$ and $\tilde a^{1j}=a^{1j}+a^{j1}$ for $j=2,\ldots,d$, $\tilde a^{ij}=a^{ij}$ for the other $(i,j)$, and $\tilde \cL$ be the corresponding operator.
Then $\tilde \cL \in \bL_1$, and since $a^{ij}$ satisfy \eqref{eq5.22} and $a^{d1}=0$, from \eqref{eq9.23} it follows that $u$ satisfies
$$
-u_t+\tilde\cL u=0 \quad \text{in}\,\,Q_2^k
$$
with the Dirichlet boundary condition $u = 0$ on $(-4,0)\times (B_2\cap \partial\Xi^k)$. Then
By Lemma \ref{lem01}, $w: = D_1u$ is in $\cH_{2}^1(Q_{3/2}^k)$ and satisfies
$$
-w_t + \tilde \cL w = 0
$$
in $Q_{3/2}^k$ with the conormal derivative boundary condition $\tilde a^{1j} D_j w = 0$ on $(-9/4,0) \times \Gamma_{3/2}^{k,1}$ and the Dirichlet boundary condition $w = 0$ on $(-9/4,0) \times \Gamma_{3/2}^{k,i}$, $i = 2, \ldots, k$.
Then again by the De Giorgi--Nash--Moser estimate, we obtain \eqref{eq9.15}.

For $\lambda > 0$, we use an idea by S. Agmon.
Set
$$
v(t,x,z) = u(t,x) \cos (\sqrt\lambda z),
$$
where $z \in \bR$.
Also set $\fB_r = \{(x,z) \in \bR^{d+1}: |x|^2 + z^2 < r^2 \}$ and
\begin{align*}
\fQ_2^k &= (-4,0) \times \left(\fB_2 \cap \Xi_{d+1}^k \right)\\
&= (-4,0) \times \left( \fB_2 \cap
\{(x,z) \in \bR^{d+1}: x_1, \ldots, x_k \in \bR^+\}\right).
\end{align*}
Then $v$ satisfies the Dirichlet boundary condition $v = 0$ on $(-4,0)\times (\fB_2\cap \partial\Xi^k_{d+1})$.
Moreover, as a function defined on $\fQ_2^k\subset \bR^{d+2}$, $v$ satisfies
$$
-v_t + \hat{\cL} v = 0
$$
in $\fQ_2^k$, where
$$
\hat{\cL} v = \sum_{i,j=1}^d D_i(a^{ij} D_j v) + D_{zz} v.
$$
Then by the above result for $\lambda = 0$ (we also need to interchange $x_d$ and $z$ coordinates), we obtain
$$
[D_i v]_{C^\alpha(\fQ_1^k)}
\le N \|D_i v\|_{L_2(\fQ_2^k)},
$$
which implies the desired inequality in the lemma.
\end{proof}

Next we derive a similar H\"older estimate for solutions satisfying the conormal derivative boundary condition on one face of the boundary and the Dirichlet boundary condition on the other faces.

\begin{lemma}
							\label{lem04}
Let $\cL \in \bL_2$, $\lambda \ge 0$, $k \in [1,d]$ be an integer, and $u \in \cH^1_2(Q_2^k)$.
Suppose that $u$ satisfies
$$
-u_t + \cL u  - \lambda u= 0
$$
in $Q_2^k$
with the conormal derivative boundary condition $a^{11}D_1u = 0$ on $(-4, 0) \times \Gamma_2^{k,1}$
and the Dirichlet boundary condition $u = 0$ on $(-4,0) \times \Gamma_2^{k,i}$, $i= 2, \ldots,k$.
Then we have
$$
[D_1 u]_{C^\alpha(Q_1^+)}
\le N \|D_1 u\|_{L_2(Q_2^+)},
$$
where $\alpha = \alpha(d,\delta) \in (0,1)$
and $N = N(d,\delta)$ are constants.
\end{lemma}

\begin{proof}
We first prove the case $\lambda = 0$.
By Lemma \ref{lem02},
$w:= D_1u$ is in $\cH^1_{2}(Q_{3/2}^k)$ and satisfies
$$
-w_t + \cL w  = 0
$$
in $Q_{3/2}^k$ with the Dirichlet boundary condition $w = 0$ on $(-9/4,0) \times \left(B_{3/2} \cap \partial \Xi^k \right)$.
Then by the De Giorgi--Nash--Moser estimate we have
$$
[w]_{C^\alpha(Q_1^+)} \le N \|w\|_{L_2(Q_2^+)}
$$
for some $\alpha = \alpha(d,\delta) \in (0,1)$ and $N=N(d,\delta)$.
For the case $\lambda > 0$, we follow the same steps as in the proof of Lemma \ref{lem44}.
\end{proof}

The following is a consequence of an interior De Giorgi--Nash--Moser H\"{o}lder estimate and the standard difference quotient argument.

\begin{lemma}
								\label{lem05}
Let $\lambda \ge 0$, $k \in [1,d]$ be an integer, and $Q_2(t_0, x_0) \subset \bR \times \Xi^k$.
Suppose that $u \in \cH^1_2(Q_2(t_0, x_0))$ satisfies
$$
-u_t + \cL u -\lambda u = 0
$$
in $Q_2(t_0, x_0)$.
Then there exist constants $\alpha = \alpha(d,\delta) \in (0,1)$
and $N = N(d,\delta)$ such that
$$
[D_i u]_{C^\alpha(Q_1(t_0, x_0))}
\le N \|D_i u\|_{L_2(Q_2(t_0, x_0))},
\quad
i = 1, \ldots, d-1.
$$
\end{lemma}

Using the classical $L_2$-estimates and the H\"{o}lder estimates proved above, we obtain the following mean oscillation estimate for $D_1 u$ when $\cL$ belongs to $\bL_2$.

\begin{lemma}
							\label{lem5.6}
Let $\cL \in \bL_2$, $\lambda > 0$, 
$k\in [1,d]$ be an integer,
and $f=(f_1,\ldots,f_d), g \in L_{2,\text{loc}}(\bR \times \overline{\Xi^k})$.
Suppose that $u \in \cH_{2,\text{loc}}^1(\bR \times \overline{\Xi^k})$ satisfies
$$
- u_t+\cL u -\lambda u= \Div f + g
$$
locally in $\bR \times \Xi^k$
with the conormal derivative boundary condition $a^{11}D_1u = f_1$ on $\bR \times \Gamma^{k,1}$ and the Dirichlet boundary condition $u=0$ on $\bR\times \Gamma^{k,i}$, $i=2,\ldots,k$.
Then, for any $r> 0$, $\kappa \ge 32d$, and $(t_0, x_0) \in \bR \times\overline{\Xi^k}$,
we have
\begin{align}
                            \label{eq9.39}
&\dashint_{Q_r^k(t_0,x_0)}
\dashint_{Q_r^k(t_0,x_0)}
\left| D_1 u(t,x) - D_1 u(s,y) \right|^2 \, dx \, dt \, dy \, ds\nonumber\\
&\le N \kappa^{-2\alpha} \dashint_{Q_{\kappa r}^k(t_0,x_0)} |D_1u|^2 \, dx \, dt
+ N \kappa^{d+2} \dashint_{Q_{\kappa r}^k(t_0,x_0)} \left(|f|^2 + \lambda^{-1}|g|^2\right) \, dx \, dt,
\end{align}
where $\alpha=\alpha(d,\delta) \in (0,1)$ and $N=N(d,\delta)>0$.
\end{lemma}

\begin{proof}
By a dilation argument, it is enough to prove the lemma only for $r = 8d/\kappa$.
Fix $(t_0,x_0) \in \bR \times \overline{\Xi^k}$.
Due to the possibility of shifting the coordinates, we may assume that $t_0 = 0$ and
$$
x_0 = (x_{0,1}, x_{0,2}, \ldots, x_{0,d}) = (x_{0,1},\ldots,x_{0,k}, 0, \ldots, 0).
$$
For $i=1,\ldots,d$, let
$$
y_{0,i}=\left\{
          \begin{array}{cl}
            x_{0,i} & \hbox{if $x_{0,i}\ge 1$,} \\
            0 & \hbox{otherwise,}
          \end{array}
        \right.
\quad r_i=\left\{
            \begin{array}{cl}
              1/8 & \hbox{if $x_{0,i}\ge 1$,} \\
              1  & \hbox{otherwise.}
            \end{array}
          \right.
$$
Then define $\hat{Q}_r(t_0,y_0)=(-r^2,0)\times \Theta_r(y_0)$, where
$$
\Theta_r(y_0) = \prod_{i=1}^k \big((y_{0,i}-rr_i)_+,y_{0,i}+ r r_i\big) \times \prod_{i=k+1}^d (-r r_i, r r_i).
$$
Clear $y_0=(y_{0,1},\ldots,y_{0,d})\in \overline{\Xi^k}$. Since $\kappa r=8d$ and $\kappa\ge 32d$,
we see that
$$
Q_r^k(t_0,x_0)\subset \hat Q_2(t_0,y_0)
\subset \hat Q_6(t_0,y_0)\subset Q^k_{\kappa r}(t_0,x_0).
$$

Take an infinitely differentiable function $\eta(t,x)$ defined on $\bR^{d+1}$ such that
$$
\eta = 1
\quad
\text{on}
\,\,
\hat{Q}_4(t_0, y_0),
\quad
\eta = 0
\quad
\text{on}
\,\,
\bR^{d+1} \setminus (-36,36) \times C_6(y_0).
$$
Here we denote
$$
C_r(y_0) = \prod_{i=1}^k \big(y_{0,i}-rr_i,y_{0,i}+rr_i\big) \times \prod_{i=k+1}^d (-rr_i, rr_i).
$$
Using Theorem \ref{thm0411_1}, we find a unique $w \in \cH_2^1(\bR_0 \times \Xi_k)$ to the equation
$$
-w_t + \cL w - \lambda w = \Div (\eta f) + \eta g
$$
in $\bR_0 \times \Xi^k$ with the conormal derivative boundary condition $a^{11}D_1u = \eta f_1$ on $\bR_0 \times \Gamma^{k,1}$ and the Dirichlet boundary condition $u = 0$ on $\bR_0 \times \Gamma^{k,i}$, $i = 2, \ldots, k$.
Moreover, the function $w$ satisfies
$$
\|Dw\|_{L_2(\bR_0 \times \Xi_k)}
\le N \|\eta f\|_{L_2(\bR_0 \times \Xi_k)}
+ N\lambda^{-1/2}\|\eta g\|_{L_2(\bR_0 \times \Xi_k)},
$$
where $N=N(d,\delta)$. From this we obtain
\begin{equation}
							\label{eq0423_2}
\dashint_{Q_r^k(t_0,x_0)}|D_1w|^2 \, dx \, dt
\le N \kappa^{d+2} \dashint_{Q_{\kappa r}^k(t_0,x_0)} \left(|f|^2+\lambda^{-1}|g|^2\right) \, dx \, dt,
\end{equation}
\begin{equation}
							\label{eq0423_3}
\dashint_{Q_{\kappa r}^k(t_0,x_0)}|D_1w|^2\, dx \, dt
\le N \dashint_{Q_{\kappa r}^k(t_0,x_0)} \left(|f|^2+\lambda^{-1}|g|^2\right) \, dx \, dt.
\end{equation}

Set $v:= u - w \in \cH_{2, \text{loc}}^1(\bR \times \overline{\Xi^k})$. Then $v$ satisfies
$$
-v_t + \cL v - \lambda v = 0
$$
in $\hat Q_4(t_0,y_0)$  with the conormal derivative boundary condition $a^{11}D_1v = 0$ on $ (-16,0) \times \Gamma_4^{k,1}$ if $x_{0,1}<1$, and the Dirichlet boundary condition $v = 0$ on $ (-16,0) \times \Gamma_4^{k,i}$  if $x_{0,i} < 1$, $i = 2, \ldots, k$.
Depending on the location of $(t_0, x_0)$, we notice that $v$ satisfies appropriate boundary conditions stated in Lemma \ref{lem44}, Lemma \ref{lem04} or Lemma \ref{lem05}.
For example, if $x_{0,i} < 1$ for $i=1,2$ and $x_{0,i} \ge 1$ for $i=3,\ldots,k$, then $v$ satisfies the boundary conditions in Lemma \ref{lem04} with $k=2$.
If $x_{0,1} \ge 1$ and $x_{0,i} < 1$ for $i=2,\cdots,k$, then by reordering the coordinates $(x_1,x_2,\ldots,x_k)\to (x_2,\ldots,x_k,x_1)$, $v$ satisfies the Dirichlet boundary conditions in Lemma \ref{lem44} with $k-1$ in place of $k$.
If $x_{0,i} \ge 1$ for all $i=1,\cdots,k$, we do not consider any boundary conditions as in Lemma \ref{lem05}.
Lemmas \ref{lem44}, \ref{lem04}, and \ref{lem05} together with a covering argument as well as a scaling argument prove that
$$
[D_1 v]_{C^\alpha(\hat Q_2(t_0,y_0))}
\le N \|D_1v\|_{L_2(\hat Q_4(t_0,y_0))}.
$$
Then we note that
\begin{multline}
							\label{eq0423_1}
\dashint_{Q_r^k(t_0,x_0)}
\dashint_{Q_r^k(t_0,x_0)}
\left| D_1 v(t,x) - D_1 v(s,y) \right|^2 \, dx \, dt \, dy \, ds
\\
\le N r^{2\alpha} [D_1 v]_{C^\alpha(\hat Q_2(t_0,y_0))}^2
\le N \kappa^{-2\alpha} \dashint_{Q_{\kappa r}^k(t_0,x_0)} |D_1 v|^2 \, dx \, dt,
\end{multline}
where in the last inequality we used the fact that $r = 8d/\kappa$.
Since $u=v+w$, we have
\begin{align*}
&\dashint_{Q_r^k(t_0,x_0)}
\dashint_{Q_r^k(t_0,x_0)}
\left| D_1 u(t,x) - D_1 u(s,y) \right|^2 \, dx \, dt \, dy \, ds\\
&\le \dashint_{Q_r^k(t_0,x_0)}
\dashint_{Q_r^k(t_0,x_0)}
\left| D_1 v(t,x) - D_1 v(s,y) \right|^2 \, dx \, dt \, dy \, ds\\
&\quad+ N \dashint_{Q_r^k(t_0,x_0)} |D_1 w|^2 \, dx \, dt \, dy \, ds :=I_1 + I_2.
\end{align*}
For $I_1$,  using $u=v+w$ again, from \eqref{eq0423_1} and \eqref{eq0423_3} we obtain
$$
I_1 \le N\kappa^{-2\alpha} \dashint_{Q_{\kappa r}^k(t_0,x_0)} |D_1u|^2 \, dx \, dt
+ N\kappa^{-2\alpha} \dashint_{Q_{\kappa r}^k(t_0,x_0)} \left(|f|^2+\lambda^{-1}|g|^2\right) \, dx \, dt.
$$
For $I_2$, we use \eqref{eq0423_2}.
By combining the estimates for $I_1$ and $I_2$, we arrive at the desired estimate \eqref{eq9.39}.
\end{proof}

Proceeding as in the proof of Lemma \ref{lem5.6}, from Theorem \ref{thm0411_1} and Lemmas \ref{lem44} and \ref{lem05} we derive the following lemma, where $\cL$ is not necessarily in $\bL_1$ or $\bL_2$.

\begin{lemma}
                                        \label{lem5.34}
Let $\lambda > 0$, $k\in [1,d]$ be an integer,
and $f=(f_1,\ldots,f_d), g \in L_{2,\text{loc}}(\bR \times\overline{\Xi^k})$.
Suppose that $a^{dj}=0$ for $j=1,2,\ldots,l$, where $l = \min\{k,d-1\}$, and $u \in \cH_{2, \text{loc}}^1(\bR \times \overline{\Xi^k})$ satisfies
\begin{equation}
                        \label{eq6.01}
-u_t+\cL u -\lambda u = \Div f + g
\end{equation}
locally in $\bR \times \Xi^k$
with the Dirichlet boundary condition on $\bR\times \partial \Xi^k$.
Then, for any $r> 0$, $\kappa \ge 32d$, and $(t_0, x_0) \in \bR \times\overline{\Xi^k}$,
we have \eqref{eq9.39} with $D_iu$, $i=1,\cdots,d-1$, in place of $D_1u$.
\end{lemma}

\subsection{$\cH^1_p$-estimates}
We consider a filtration of dyadic cubes $\{\bC_l, l \in \bZ\}$ in $\bR \times \Xi^k$, where $\bZ = \{0, \pm 1, \pm 2, \ldots\}$.
The set $\bC_l$ is the collection of dyadic parabolic cubes in $\bR \times \Xi^k$ of the form
$$
(i_0 2^{-2l}, (i_0+1)^{-2l}] \times (i_12^{-l}, (i_1+1)2^{-l}]
\times \ldots \times (i_d2^{-l}, (i_d+1)2^{-l}],
$$
where $i_0, i_1, \ldots, i_d \in \bZ$ and $i_1, \ldots, i_k \ge 0$.
Let $\cC$ be the collection of the dyadic cubes in $\bC_l$ for all $l \in \bZ$.
If $(t_1,x_1) \in C \in \cC$, then there exist the smallest $r>0$ and $(t_0,x_0) \in \bR \times\overline{\Xi^k}$ such that $C \subset Q^k_r(t_0,x_0)$ and
\begin{multline}
							\label{eq1213}
\dashint_C\dashint_C |f(t,x) - f(s,y)| \, dx \, dy \, dt \, ds
\\
\le N(d) \dashint_{Q_r^k(t_0,x_0)}\dashint_{Q_r^k(t_0,x_0)} |f(t,x) - f(s,y)| \, dx \, dy \, dt \, ds.
\end{multline}
On the other hand, for any $(t_0,x_0) \in \bR \times \overline{\Xi^k}$ and $r>0$,
if $C \in \cC$ is the smallest cube containing $Q_r^k(t_0,x_0)$, then
\begin{equation}
							\label{eq1214}
\dashint_{Q_r^k(t_0,x_0)} |f(t,x)| \, dx \, dt
\le N(d) \dashint_C |f(t,x)| \, dx \, dt.
\end{equation}

For a function $f \in L_{1,\text{loc}}(\bR \times \Xi^k)$, the maximal and sharp functions of $f$ in our context are given by
\begin{align*}
\cM f(t,x) &= \sup_{(t,x) \in C, C \in \cC} \dashint_C |f(s,y)| \, dy \, ds,\\
f^{\#}(t,x) &= \sup_{(t,x) \in C, C \in \cC} \dashint_C \dashint_C|f(s,y)-f(r,z)| \, dy \, dz \, ds \, dr.
\end{align*}
For $p \in (1,\infty)$, by the Fefferman--Stein theorem on sharp functions and the Hardy--Littlewood maximal function theorem, we have
\begin{equation}
							\label{eq1215}
\|f\|_{L_p(\bR \times \Xi^k)} \le N(d,p) \|f^{\#}\|_{L_p(\bR \times \Xi^k)},
\end{equation}
\begin{equation}
							\label{eq1216}
\|\cM f\|_{L_p(\bR \times \Xi^k)} \le N(d,p) \|f\|_{L_p(\bR \times \Xi^k)}.
\end{equation}

Notice that, in the proposition below, the functions $f$ and $g$ have compact supports in $\bR \times \Xi^k$.

\begin{proposition}
                                    \label{prop5.8}
Let $\cL \in \bL_2$, $\lambda > 0$, $p\in (2,\infty)$, $k\in [1,d]$ be an integer, $f=(f_1,\ldots,f_d), g \in C_0^{\infty}(\bR \times \Xi^k)$.
Then there exists a unique solution $u\in \cH^1_2(\bR\times \Xi^k)$ to the equation \eqref{eq6.01} with the conormal derivative boundary condition $a^{11} D_1 u = f_1$ on $\bR \times \Gamma^{k,1}$ and the Dirichlet boundary condition $u=0$ on $\bR\times \Gamma^{k,i}$, $i=2,\ldots,k$. Moreover, we have \begin{equation}
                                            \label{eq6.00b}
\|D_1 u\|_{L_p(\bR \times  \Xi^k)}
\le N \|f\|_{L_p(\bR \times \Xi^k)} + N \lambda^{-1/2} \|g\|_{L_p(\bR \times \Xi^k)},
\end{equation}
where $N=N(d,\delta,p)$.
\end{proposition}

\begin{proof}
The first part of the proposition is due to Theorem \ref{thm0411_1}.
For the second part, we note that $\Div f+g\in C_0^\infty(\bR\times\Xi^k)$ and $a^{11}D_1 u=f_1=0$ on $\bR\times\Gamma^{k,1}$.
Thus, by Lemma \ref{lem02} (in this case we use the domain $\bR \times \Xi^k$ and $\varphi \equiv 1$ in the proof of Lemma \ref{lem02}) $w := D_1 u$ is in $\cH_2^1(\bR \times \Xi^k)$ and satisfies
$$
- w + \cL w - \lambda w = D_1 \left( \Div f + g \right)
$$
in $\bR \times \Xi^k$ with the Dirichlet boundary condition $w = 0$ on $\bR \times \partial \Xi^k$. Then the De Giorgi--Nash--Moser estimate implies that
$D_1 u\in L_\infty(\bR \times\overline{\Xi^k})$. In particular, we have $D_1 u\in L_p(\bR \times \Xi^k)$ by H\"older's inequality.

Now we are ready to derive \eqref{eq6.00b}.
For each $(t_1,x_1) \in \bR \times \Xi^k$ and $C \in \cC$ such that $(t_1,x_1) \in C$, we find the smallest $r>0$ and $(t_0,x_0) \in \bR \times\overline{\Xi^k}$ satisfying $C \subset Q_r(t_0,x_0)$ and \eqref{eq1213}.
Then Lemma \ref{lem5.6} together with \eqref{eq1213} implies that
\begin{align*}
I&:= \dashint_C \dashint_C |D_1 u(t,x)-D_1 u(s,y)| \, dx \, dt \, dy \, ds\\
&\le N \kappa^{-\alpha} \left(\dashint_{Q^k_{\kappa r}(t_0,x_0)} |D_1 u|^2 \, dx \, dt\right)^{1/2}
+ N \kappa^{\frac{d}{2}+1}\left(\dashint_{Q^k_{\kappa r}(t_0,x_0)} |h|^2 \, dx \, dt\right)^{1/2}.
\end{align*}
Here and below, $h := |f| + \lambda^{-1/2}|g|$.
Since $(t_1,x_1) \in Q^k_{\kappa r}(t_0,x_0)$, due to \eqref{eq1214} and the definition of maximal functions, from the above inequality we have
$$
I \le N \kappa^{-\alpha} \left(\cM|D_1 u|^2(t_1,x_1)\right)^{1/2} + N \kappa^{\frac{d}{2}+1} \left(\cM|h|^2(t_1,x_1)\right)^{1/2}.
$$
This along with the fact that $C$ is an arbitrary parabolic cube containing $(t_1,x_1)$ proves that
$$
(D_1 u)^{\#}(t_1,x_1) \le N \kappa^{-\alpha} \left(\cM|D_1 u|^2(t_1,x_1)\right)^{1/2} + N \kappa^{\frac{d}{2}+1} \left(\cM|h|^2(t_1,x_1)\right)^{1/2}.
$$
Since $D_1 u \in L_p(\bR \times \Xi^k)$, we can apply \eqref{eq1215} and \eqref{eq1216} to the above inequality and get
$$
\|D_1 u\|_{L_p(\bR\times \Xi^k)}\le N\kappa^{-\alpha}\|D_1 u\|_{L_p(\bR\times \Xi^k)}+N \kappa^{\frac{d}{2}+1}\|D_1 h\|_{L_p(\bR\times \Xi^k)}.
$$
Then by choosing $\kappa > 32$ sufficiently large so that $N \kappa^{-\alpha} \le 1/2$, we obtain \eqref{eq6.00b}.
The proposition is proved.
\end{proof}

Note that if a given equation is as in \eqref{eq6.01}, the estimate \eqref{eq6.00b} is not enough to get complete $L_p$-estimates as in \eqref{eq518_3} for $p \in (1,2)$ using the duality argument.
Nevertheless, one can still get \eqref{eq518_3} if the right-hand side of equation is in a particular form as in \eqref{eq518_4}. Indeed,
by using the duality argument, from Proposition \ref{prop5.8} we deduce the following corollary.

\begin{corollary}
                            \label{cor4.11}
Let $\cL \in \bL_1$, $\lambda > 0$, $p\in (1,2)$, $k\in [1,d]$ be an integer,  $ f \in C_0^{\infty}(\bR \times\overline{\Xi^k})$, and $u \in \cH_2^1(\bR\times\Xi^k) \cap C_{\text{loc}}^{\infty}(\bR \times\overline{\Xi^k})$.
If $u$ satisfies
\begin{equation}
							\label{eq518_4}
-u_t+\cL u -\lambda u = D_1 f
\end{equation}
with the conormal derivative boundary condition $a^{1j} D_j u = f$ on $\bR \times \Gamma^{k,1}$ and the Dirichlet boundary condition $u=0$ on $\bR\times \Gamma^{k,i}$, $i=2,\ldots,k$, then we have $u\in \cH_p^1(\bR\times\Xi^k)$ and
\begin{equation}
							\label{eq518_3}
\sqrt\lambda\|u\|_{L_p(\bR \times  \Xi^k)}+\|D u\|_{L_p(\bR \times  \Xi^k)}
\le N \| f\|_{L_p(\bR \times \Xi^k)},
\end{equation}
where $N=N(d,\delta,p)$.
\end{corollary}

Using Lemma \ref{lem5.34} and a scaling argument, we prove the following proposition, where $\cL$ is not necessarily in $\bL_1$ or $\bL_2$.

\begin{proposition}
							\label{prop5.7}
Let $\lambda > 0$, $k\in [1,d]$ be an integer, $f=(f_1,\ldots,f_d), g \in C_0^{\infty}(\bR \times\overline{\Xi^k})$, and
$u \in \cH_p^1(\bR\times\Xi^k) \cap C_{\text{loc}}^{\infty}(\bR \times\overline{\Xi^k})$.
Suppose that either $a^{dj}=0, j=1,2,\ldots, l$ and $p\in [2,\infty)$ or $a^{id}=0, i=1,2,\ldots,l$ and $p\in (1,2]$, where $l = \min\{k, d-1\}$.
If $u$ satisfies \eqref{eq6.01} in $\bR\times \Xi^k$
with the Dirichlet boundary condition $u=0$ on $\bR\times \partial \Xi^k$, then we have
\begin{equation}
                                            \label{eq6.00}
\sqrt\lambda\|u\|_{L_p(\bR \times \Xi^k)}+\|D u\|_{L_p(\bR \times \Xi^k)}
\le N \|f\|_{L_p(\bR \times \Xi^k)} + N \lambda^{-1/2} \|g\|_{L_p(\bR \times \Xi^k)},
\end{equation}
where $N=N(d,\delta,p)$.
\end{proposition}

\begin{proof}
The case when $p=2$ follows from Theorem \ref{thm0411_1}. In the sequel, we assume that $p\neq 2$. By the duality argument, we may further assume that $p\in (2,\infty)$.
Then by following the proof of \eqref{eq6.00b}, from Lemma \ref{lem5.34} we obtain
\begin{equation}
							\label{eq0423_4}
\sum_{i=1}^{d-1}\| D_i u \|_{L_p(\bR \times \Xi^k)}
\le N \|f\|_{L_p(\bR \times \Xi^k)} + N \lambda^{-1/2}\|g\|_{L_p(\bR \times \Xi^k)}.
\end{equation}

Now we prove that \eqref{eq0423_4} implies \eqref{eq6.00}.
To do this, we use $L_p$-estimates for equations defined in $\bR \times \bR^d$ and an idea of scaling in \cite{DK11}. For an $\varepsilon>0$ to be chosen later, introduce
$$
v(t,x)=u\left(t,\frac{x_1}{\varepsilon},\frac{x_2}{\varepsilon}, \ldots, \frac{x_{d-1}}{\varepsilon}, x_d\right) := u(t, x'/\varepsilon,x_d),
$$
where we denote $x'=(x_1,x_2,\ldots,x_{d-1})$.
From \eqref{eq6.01} we see that $v$ satisfies
$$
-v_t+\Delta_{d-1} v + D_d(a^{dd}D_d v) - \lambda v =\Div \hat f + \hat g
$$
in $\bR\times\Xi^k$ with the Dirichlet boundary condition $v=0$ on $\bR\times \partial\Xi^k$.
Here
\begin{align*}
&\hat f_i(t,x) = \varepsilon f_i(t,x'/\varepsilon, x_d) +D_i v - \varepsilon^2 \sum_{j=1}^{d-1}a^{ij} D_j v - \varepsilon a^{id} D_d v,
\quad
i = 1, \ldots, d-1,\\
&\hat f_d(t,x)= f_d(t,x'/\varepsilon, x_d)
- \varepsilon \sum_{j=1}^{d-1} a^{dj} D_j v,\\
&\hat g (t,x) = g(t,x'/\varepsilon, x_d).
\end{align*}
Since $v$ has the Dirichlet boundary condition $v=0$ on $\bR\times \partial \Xi^k$, using a similar extension process presented in the proof of Theorem \ref{thm0411}, we obtain $\bar v \in \cH_p^1(\bR\times\bR^d)$ and $\bar f, \bar g \in L_p(\bR \times \bR^d)$ such that $\bar v$, $\bar f$, and $\bar g$ are the extensions of $v$, $\hat f$, and $\hat g$ to the domain $\bR \times \bR^d$ and satisfy
$$
-\bar{v}_t + \Delta_{d-1} \bar{v} + D_d(a^{dd}D_d \bar{v}) - \lambda \bar{v} = \Div \bar{f} + \bar g,
$$
in $\bR \times \bR^d$, where the $L_p$ norms of $\bar v$, $D \bar{v}$,  $\bar f$, and $\bar g$ in $\bR \times \bR^d$ are comparable to those of $v$, $Dv$, $\hat f$, and $\hat g$ in $\bR \times  \Xi^k$, respectively.
In fact, to extend the equation in $\bR \times \Xi^k$ to the one defined in $\bR \times \Xi^{k-1}$, we take the odd extensions of $v$, $\hat f_i$, and $\hat g$ with respect to $x_k$ except $\hat f_k$, for which we take the even extension with respect to $x_k$.
Notice that the coefficient $a^{dd}$ is a measurable function of only time or $x_d$.
Thus by the $L_p$-estimates proved in  \cite[Theorem 5.1 (iii)]{DK11}\footnote{Indeed, this theorem is applicable whenever $a^{ij}$ are measurable functions of time and one spatial variable, say $x_k$, except $a^{kk}$ which is a measurable function of either time or $x_k$.} together with the definitions of $\hat f$ and $\hat g$, and the comparability of the $L_p$ norms of $v$, $Dv$, $\hat f$, and $\hat g$ with those of $\bar{v}$, $\bar f$, and $\bar g$, we obtain
\begin{multline*}
\sqrt{\lambda}\|v\|_{L_p(\bR \times  \Xi^k)} + \|Dv\|_{L_p(\bR \times  \Xi^k)}
\le \sqrt{\lambda}\|\bar v\|_{L_p(\bR \times\bR^d)}
+ \|D \bar v\|_{L_p(\bR \times\bR^d)}\\
\le N \|\bar f\|_{L_p(\bR \times\bR^d)} + N \lambda^{-1/2}\|\bar g\|_{L_p(\bR \times\bR^d)}
\le N \|\hat f\|_{L_p(\bR \times \Xi^k)} + N \lambda^{-1/2}\|\hat g\|_{L_p(\bR \times\Xi^k)}\\
\le N(d,\delta,p,\varepsilon) \|f\|_{L_p(\bR \times \Xi^k)}
+ N(d,\delta,p,\varepsilon) \lambda^{-1/2}\|g\|_{L_p(\bR \times \Xi^k)}\\
+ N(d,\delta,p,\varepsilon) \sum_{i=1}^{d-1}\|D_iv\|_{L_p(\bR \times \Xi^k)}
+ N(d,\delta,p) \varepsilon \|D_d v\|_{L_p(\bR \times \Xi^k)}.
\end{multline*}
From these inequalities with an appropriate choice of $\varepsilon$ so that $N(d,\delta,p) \varepsilon < 1/2$, we get
\begin{align*}
\sqrt{\lambda}\|v\|_{L_p(\bR \times  \Xi^k)} + \|Dv\|_{L_p(\bR \times  \Xi^k)}
&\le N \|f\|_{L_p(\bR \times \Xi^k)}
+ N\lambda^{-1/2}\|g\|_{L_p(\bR \times \Xi^k)}\\
&\quad+ N \sum_{i=1}^{d-1}\|D_iv\|_{L_p(\bR \times \Xi^k)}.
\end{align*}
We now scale back to $u$ and use \eqref{eq0423_4} to get \eqref{eq6.00}. The proposition is proved.
\end{proof}

\section{Proofs of Theorems \ref{thm01} and \ref{thm02}}
                                    \label{sec4}

We complete the proofs of Theorems \ref{thm01} and \ref{thm02} in this section.

\begin{proof}[Proof of Theorem \ref{thm01}]
We prove the theorem only when $T = \infty$. The case $T < \infty$ is deduced from this case and the standard argument (see, for example, the proof of Theorem 2.1 in \cite{Kr07}).

We first prove the a priori estimate \eqref{eq1201} for $\lambda > 0$.
For the case $\lambda = 0$, we just take the limit as $\lambda \searrow 0$.
Without loss of generality, we assume that $u \in C_0^{\infty}(\bR \times\overline{\Xi^k})$, $f \in C_0^{\infty}(\bR \times \overline{\Xi^k})$, and $a^{ij}$ are infinitely differentiable. Thanks to Remark \ref{rem2.2}, we may also assume that $k\le d-1$.
We rewrite the equation \eqref{eq01} as
$$
-u_t + \sum_{i=1}^k D_{i}^2 u+\sum_{i,j=k+1}^d a^{ij}D_{ij}u - \lambda u = f+\sum_{i=1}^k D_{i}^2 u - \sum_{i\text{ or } j\le k} a^{ij} D_{ij}u :=F.
$$
Since $u$ has the Dirichlet boundary condition $u=0$ on $\bR\times \partial \Xi^k$, as in the proof of Theorem \ref{thm0411} we find $\bar{u} \in W_p^{1,2}(\bR\times\bR^d)$ and $\bar{F} \in L_p(\bR \times \bR^d)$ such that $\bar{u}$ and $\bar{F}$ are the extensions of $u$ and $F$ to the domain $\bR \times \bR^d$ satisfying
$$
-\bar{u}_t + \sum_{i=1}^k D_{i}^2 \bar{u}+\sum_{i,j=k+1}^d a^{ij}D_{ij}\bar{u} - \lambda \bar{u} = \bar{F},
$$
in $\bR \times \bR^d$ and
\begin{equation}
                            \label{eq1.09}
\|\bar u\|_{W_p^{1,2}(\bR \times \bR^d)} \cong \|u\|_{W_p^{1,2}(\bR \times \Xi^k)},
\quad
\|\bar F\|_{L_p(\bR \times \bR^d)} \cong \|F\|_{L_p(\bR \times \Xi^k)}.
\end{equation}
Now we observe that, for $i,j = k+1, \ldots, d,(i,j)\neq (d,d)$, the coefficients $a^{ii}$ are measurable functions of time and one spatial variable, and the coefficient $a^{dd}$ is a function of only time or one spatial variable.
Then from the $L_p$-estimates for equations in the whole space established in \cite[Theorem 2.2 (iii) and Theorem 2.3 (iii)]{Dong12} together with the definition of $\bar{F}$, we obtain
\begin{multline}
                                \label{eq4.15}
\lambda \|\bar{u}\|_{L_p(\bR \times \bR^d)} + \lambda^{1/2}\|D\bar{u}\|_{L_p(\bR \times \bR^d)} + \|D^2\bar{u}\|_{L_p(\bR \times\bR^d)}
+ \|\bar{u}_t\|_{L_p(\bR \times \bR^d)}\\
\le N(d,\delta) \|\bar{F}\|_{L_p(\bR \times \bR^d)}
\le N \sum_{i\text{ or } j\le k} \|D_{ij} u\|_{L_p(\bR \times \Xi^k)} + N \|f\|_{L_p(\bR \times \Xi^k)},
\end{multline}
which, together with \eqref{eq1.09}, implies that
\begin{multline*}
\lambda \|u\|_{L_p(\bR\times\Xi^k)}+ \lambda^{1/2} \|Du\|_{L_p(\bR\times\Xi^k)}+
\|D^2 u\|_{L_p(\bR\times\Xi^k)} + \|u_t\|_{L_p(\bR\times\Xi^k)}\\
\le N \|f\|_{L_p(\bR \times\Xi^k)} + N\sum_{i\text{ or } j\le k}\|D_{ij}u\|_{L_p(\bR \times\Xi^k)}.
\end{multline*}
Thus, by symmetry, the estimate \eqref{eq1201} follows once we have
\begin{equation}
							\label{eq0424_5}
\|DD_{1}u\|_{L_p(\bR\times\Xi^k)}
\le N(d,\delta,p) \|f\|_{L_p(\bR\times\Xi^k)}.
\end{equation}
To prove this estimate we consider two cases as below.

{\em (i) The case $a^{dd}=a^{dd}(x_d)$}: Denote $y' = (y_1,\ldots,y_{d-1})$ and let
$$
y_d := \phi(x_d) = \int_0^{x_d} \frac{1}{a^{dd}(s)} \, ds,
\quad
y_i := x_i,
\quad
i=1,\ldots,d-1.
$$
We set
$$
\tilde u(t,y) = u(t,y', \phi^{-1}(y_d)),
\quad
\tilde f(t,y) = f(t, y', \phi^{-1}(y_d)),
$$
$$
\tilde{a}^{1j}(t,y_d) = (a^{1j} + a^{j1})(t,\phi^{-1}(y_d)),
\quad
\tilde{a}^{j1} = 0,
\quad
j = 2, \ldots, d-1,
$$
$$
\tilde{a}^{jd}(t,y_d) = \frac{a^{dj}+a^{jd}}{a^{dd}}\left(t,\phi^{-1}(y_d)\right),\quad \tilde{a}^{dj} = 0,
\quad j = 1,\ldots,d-1,
$$
$$
\tilde{a}^{dd}(y_d) = 1/a^{dd}(\phi^{-1}(y_d)),
$$
and
$\tilde{a}^{ij}(t,y_d) = a^{ij}(t,\phi^{-1}(y_d))$ for the other $(i,j)$.
Then the operator $\cL$ defined by
$$
\cL \tilde u = D_i\left(\tilde{a}^{ij} D_j \tilde u \right)
$$
belongs to $\bL_1$ and is uniformly non-degenerate with an ellipticity constant depending only on $\delta$.
A simple calculation shows that $\tilde u$ satisfies
$$
-\tilde{u}_t + \cL \tilde u - \lambda \tilde u = \tilde f
$$
with the Dirichlet boundary condition $\tilde u = 0$ on $\bR\times \partial \Xi^k$.
By Lemma \ref{lem01} $w := D_1\tilde u$ satisfies
$$
-w_t + \cL w - \lambda w= D_1 \tilde f
$$
in $\bR \times \Xi^k$ with the conormal derivative boundary condition
$a^{1j}D_j w = \tilde f $ on $\bR \times \Gamma^{k,1}$ and the Dirichlet boundary condition $w=0$ on $\bR \times \Gamma^{k,i}$, $i=2,\ldots,k$.
Then by Corollary \ref{cor4.11} applied to $w$ we arrive at
$$
\|DD_1\tilde u\|_{L_p(\bR\times\Xi^k)} = \|Dw\|_{L_p(\bR\times\Xi^k)}
\le N \|\tilde f\|_{L_p(\bR\times\Xi^k)},
$$
which implies the inequality \eqref{eq0424_5}.

{\em (ii) The case $a^{dd}=a^{dd}(t)$}: This case is actually simpler. We set
$$
\tilde{a}^{1j}(t,x_d) = (a^{1j} + a^{j1})(t,x_d),
\quad
\tilde{a}^{j1} = 0,
\quad
j = 2, \ldots, d-1,
$$
$$
\tilde{a}^{jd}(t,x_d) = (a^{dj}+a^{jd})(t,x_d),\quad \tilde{a}^{dj} = 0,
\quad j = 1, \ldots, d-1,
$$
and
$\tilde{a}^{ij}(t,x_d) = a^{ij}(t,x_d)$ for the other $(i,j)$. Then the operator $\cL$ defined by
$$
\cL u = D_i\left(\tilde{a}^{ij} D_j u \right)
$$
belongs to $\bL_1$ and is uniformly non-degenerate with an ellipticity constant depending only on $\delta$. Moreover,  $u$ satisfies
$$
-u_t + \cL u -\lambda u = f
$$
with the Dirichlet boundary condition $u = 0$ on $\bR\times \partial \Xi^k$. As in the first case, by applying Lemma \ref{lem01} and Corollary \ref{cor4.11} we get \eqref{eq0424_5}.
Therefore, we have proved the a priori estimate \eqref{eq1201}.
The solvability assertion follows from the a priori estimate and the method of continuity.
The theorem is proved.
\end{proof}

\begin{proof}[Proof of Theorem \ref{thm02}]
Again we prove the case $T = \infty$ only.
As in the proof of Theorem \ref{thm01}, we assume that $u \in C_0^\infty(\bR \times \overline{\bR^d_+})$, $f \in C_0^{\infty}(\bR \times \overline{\Xi^k})$, and $a^{ij}$ are infinitely differentiable. Proceeding as in the proof of Theorem \ref{thm01}, by the results in \cite{Dong12} we obtain \eqref{eq4.15} with $k=1$, where $\bar u$ and $\bar F$ are now the even extensions of $u$ and $F$ with respect to $x_1$, respectively.
Then, as before,  it suffices to prove \eqref{eq0424_5}.
Using the change of variables in the proof of Theorem \ref{thm01} for the case $ a^{dd} = a^{dd}(x_d)$, we see that
$w = D_1 \tilde u$ satisfies
$$
-w_t + \cL w - \lambda w = D_1 \tilde f
$$
with the Dirichlet boundary condition $w = 0$ on $\bR \times \partial \bR^d_+$.
In the case $a^{dd}=a^{dd}(t)$ we obtain the above equation with $w = D_1 u$ and $f$ in place of $\bar f$.
Then to prove \eqref{eq0424_5}, we argue as in the proof of Theorem \ref{thm01} using Proposition \ref{prop5.7} for $k=1$. The theorem is proved.
\end{proof}

\section{Equations with partially VMO coefficients}
                                                            \label{sec6}

In this section, we consider second-order parabolic equations
$$
-u_t+Lu-\lambda u:=-u_t+a^{ij}D_{ij}u+b^iD_iu+cu-\lambda u=f
$$
in $\bR_T \times \bR^d_+$ with leading coefficients $a^{ij}$ which also depend on $x' = (x_1, \ldots, x_{d-1})$. As functions of $(t,x)$, the coefficients $a^{ij}$ are supposed to be measurable with respect to $x_d$, and have small local mean oscillations in the other variables.
To be more precise, we impose the following assumption which contains a parameter $\gamma>0$ to be specified later.
\begin{assumption}                          \label{assum1}
The coefficients $a^{ij}$, $b^i$, and $c$ satisfy the following conditions.

(i) $a^{ij}$ satisfy \eqref{eq5.21}.

(ii) There is a constant $R_0\in (0,1]$ such that the following holds. For any parabolic cylinder $Q$ of radius $r\in (0,R_0)$, there exist $\bar a^{ij}=\bar a^{ij}(x_d)$, which depend on the cylinder $Q$ and satisfy \eqref{eq5.21}, such that
$$
\sum_{i,j=1}^d \dashint_{Q}|a^{ij}(t, x)-\bar a^{ij}(x_d)|\,dx\,dt \le \gamma.
$$

(iii) $b^i$ and $c$ are measurable functions bounded by a constant $K>0$.
\end{assumption}

We state the main results of this section.

\begin{theorem}[The Dirichlet problem]
                                \label{thm3}
Let $p \in (1,2]$, $T \in (-\infty,\infty]$,
and $f \in L_p(\bR_T\times\bR^d_+)$.
Then there exist constants $\gamma\in (0,1)$ and $N>0$ depending only on $d$, $\delta$, and $p$
such that under Assumption \ref{assum1}
the following hold true.
For any $u \in W_p^{1,2}(\bR_T\times\bR^d_+)$ satisfying $u=0$ on $\bR_T\times \partial \bR^d_+$ and
\begin{equation}
                             \label{eq15.16.01}
-u_t+L u - \lambda u =f
\end{equation}
in $\bR_T\times\bR^d_+$, we have
\begin{multline}
                             \label{eq15.16.02}
\lambda\|u\|_{L_{p}(\bR_T\times\bR^d_+)}+\lambda^{1/2}
\|Du\|_{L_{p}(\bR_T\times\bR^d_+)}+\|D^2u\|_{L_{p}(\bR_T\times\bR^d_+)}
+\|u_t\|_{L_{p}(\bR_T\times\bR^d_+)}\\
\le
N\|f\|_{L_{p}(\bR_T\times\bR^d_+)},
\end{multline}
provided that $\lambda \ge \lambda_0$,
where $\lambda_0 \ge 0$ is a constant
depending only on $d$, $\delta$, $p$, $K$, and $R_0$. Moreover, for any  $\lambda > \lambda_0$, there exists a unique $u \in W_p^{1,2}(\bR_T\times \bR^d_+)$ solving \eqref{eq15.16.01} with the Dirichlet boundary condition $u=0$ on $\bR_T\times \partial \bR^d_+$.
\end{theorem}

\begin{theorem}[The Neumann problem]
                                \label{thm4}
Let $p \in (2,\infty)$, $T \in (-\infty,\infty]$,
and $f \in L_p(\bR_T\times\bR^d_+)$.
Then there exist constants $\gamma\in (0,1)$ and $N>0$ depending only on $d$, $\delta$, and $p$
such that under Assumption \ref{assum1}
the following hold true.
For any $u \in W_p^{1,2}(\bR_T\times\bR^d_+)$ satisfying \eqref{eq15.16.01} and $D_1 u=0$ on $\bR_T \times \partial \bR^d_+$, we have \eqref{eq15.16.02}
provided that $\lambda \ge \lambda_0$,
where $\lambda_0 \ge 0$ is a constant
depending only on $d$, $\delta$, $p$, $K$, and $R_0$. Moreover, for any  $\lambda > \lambda_0$, there exists a unique $u \in W_p^{1,2}(\bR_T\times\bR^d_+)$ solving \eqref{eq15.16.01} with the Neumann boundary condition $D_1 u=0$ on $\bR_T\times\partial \bR^d_+$.
\end{theorem}

\subsection{Estimates of $u_t$ and $DD_{x''}u$}

Let us first fix some additional notation used for the remaining part of this paper.
We write
$$
B_r^+(x_0) := B_r(x_0) \cap \bR^d_+,
\quad
Q_r^+(t_0, x_0) := (t_0 - r^2, t_0) \times B_r^+(x_0),
$$
$$
\Gamma_r(x_0) := B_r(x_0) \cap \partial \bR_+^d.
$$
As before, we write $B_r^+$ and $Q_r^+$ if $x_0 = 0$ and $(t_0, x_0) = 0$, respectively.
Recall $x' = (x_1, \ldots, x_{d-1})$ if $d \ge 2$ and denote $x'' = (x_2, \ldots, x_{d-1})$ if $d \ge 3$.

Throughout this and the next subsection, we assume that $b^i\equiv c\equiv 0$.
We first present several estimates for $u_t$ and $D_{x''}^2 u$. For
convenience, we set $D_{x''}^2 u\equiv 0$ if $d=2$ (here $d$ is at least 2). The next lemma
is a consequence of the Krylov--Safonov estimate.

\begin{lemma}
                                     \label{lem4.6}
Let $\lambda\ge 0$, $q\in (1,\infty)$, and $r > 0$. Assume that
$a^{ij}=a^{ij}(x_d)$, $u \in C^{\infty}(\overline{Q_r^+})$ satisfies $-u_t+Lu-\lambda u=0$ in $B_{r}^+$, and either $u$ or $D_1 u$
vanishes on $ (-r^2,0) \times \Gamma_{r}$. Then there exist
constants $N = N(d, \delta,q)$ and $\alpha=\alpha(d,\delta)\in
(0,1]$ such that
\begin{multline}
                                          \label{eq9.10.1}
 [u_t]_{C^\alpha(Q_{r/2}^+)}+[D^{2}_{x''}u]_{C^\alpha(Q_{r/2}^+)}+
 \lambda [u]_{C^\alpha(Q_{r/2}^+)}\\
\le N r^{-\alpha}
\left( \dashint_{Q_r^+} |u_t|^{q}+|D^{2}_{x''}u|^{q}+\lambda^q |u|^q \, dx \, dt \right)^{1/q}.
\end{multline}
\end{lemma}
\begin{proof}
For $\lambda=0$, \eqref{eq9.10.1} directly follows from the parabolic
Krylov--Safonov estimate since $u_t$ and $D^{2}_{x''}u$ satisfy the same equation. The general case $\lambda>0$ then follows from Agmon's idea presented in the proof of Lemma \ref{lem44}.
\end{proof}

Denote $U:=|u_t|+|D_{x''}^2 u|+\lambda |u|$. From Theorems \ref{thm01}, \ref{thm02}, Lemma \ref{lem4.6}, and the corresponding interior estimates, we deduce the following mean oscillation estimate of $U$. The proof is similar to that of Lemma \ref{lem5.6} with obvious modifications, and thus omitted.

\begin{lemma}
                                            \label{lem4.7}
Let $\lambda\ge 0$, $q\in (1,2]$, $r>0$, $\kappa\ge 32$, $(t_0,x_0)\in \bR\times
\overline{\bR^d_+}$, and $f\in L_q(Q_{\kappa r}^+(t_0,x_0))$. Suppose that $a^{ij}=a^{ij}(x_d)$, and $u\in W^{1,2}_q(Q_{\kappa r}^+(t_0,x_0))$
satisfies
\begin{equation*}
-u_t+Lu-\lambda u=f
\end{equation*}
in $Q_{\kappa r}^+(t_0,x_0)$
with the Dirichlet boundary condition $u=0$ on $ \left(t_0-(\kappa r)^2, t_0\right) \times \Gamma_{\kappa r}(x_0)$. Then
\begin{multline*}
\dashint_{Q^{+}_r(t_0,x_0)}\dashint_{Q^{+}_r(t_0,x_0)}| U (t,x)-
 U (s,y)|^q\,dx\,dt\,dy\,ds\\
\leq N\kappa^{d+2} \dashint_{Q^{+}_{\kappa r}
(t_0,x_0)} |f|^q \,dx\,dt
 +N\kappa^{-q\alpha} \dashint_{Q^{+}_{\kappa r}(t_0,x_0)}
 | U|^q\,dx\,dt,
\end{multline*}
where $\alpha=\alpha(d,\delta) \in (0,1)$ and $N = N(d,\delta,q)>0$. The same
estimate holds for $q\in [2,\infty)$ if $D_1u$  vanishes on
$\left(t_0-(\kappa r)^2, t_0\right) \times \Gamma_{\kappa r}(x_0)$ instead of $u$.
\end{lemma}

Lemma \ref{lem4.7} together with a perturbation argument gives the
next result for general operators $L$ satisfying Assumption
\ref{assum1}.

\begin{lemma}
                                            \label{lem4.8}
Let $\lambda\ge 0$,  $q\in (1,2]$, $\beta\in (1,\infty)$,
$\beta'=\beta/(\beta-1)$, $(t_1,x_1)\in \bR\times \overline{\bR^d_+}$, and $f\in
L_{q,\text{loc}}(\bR\times \overline{\bR^d_+})$. Suppose $u\in
W^{1,2}_{q,\text{loc}}(\bR\times \overline{\bR^d_+})$ vanishes outside $Q^+_{R_0}(t_1,x_1)$
and satisfies
\begin{equation*}
-u_t+Lu-\lambda u=f
\end{equation*}
in $\bR\times\bR^d_+$ with the Dirichlet boundary condition $u=0$ on $\bR\times\partial\bR_+^d$.
Then under Assumption \ref{assum1}, for any $r>0$, $\kappa\ge 32$,
and $(t_0,x_0)\in \bR\times \overline{\bR^d_+}$, we have
\begin{align}
                                                  \label{osc2}
&\dashint_{Q^{+}_r(t_0,x_0)}\dashint_{Q^{+}_r(t_0,x_0)}| U(t,x)-
 U (s,y)|^q\,dx\,dt\,dy\,ds\nonumber\\
&\quad\leq N\kappa^{d+2} \dashint_{Q^{+}_{\kappa r}
(t_0,x_0)} |f|^q \,dx\,dt
 +N\kappa^{-q\alpha} \dashint_{Q^{+}_{\kappa r}(t_0,x_0)}
 |  U|^q \,dx\,dt\nonumber\\
&\quad\quad+N\kappa^{d+2} \Big(\dashint_{Q^{+}_{\kappa r}(t_0,x_0)}
 |  D^2 u|^{\beta q}\,dx\,dt\Big)^{\frac 1 {\beta}}\gamma^{\frac 1 {\beta'}},
\end{align}
where $\alpha$ is the constant from Lemma \ref{lem4.7} and the constant $N$ depends only on $d$, $\delta$, $\beta$, and
$q$. The same estimate holds for $q\in [2,\infty)$ if $D_1u$
vanishes on $\bR\times\partial\bR_+^d$ instead of $u$.
\end{lemma}
\begin{proof}
We choose $Q=Q_{\kappa r}(t_0,x_0)$ if $\kappa r<R_0$ and $Q=Q_{R_0}(t_1,x_1)$ if $\kappa r\ge R_0$. For this $Q$, let $\bar a^{ij}=\bar a^{ij}(x_d)$ be the coefficients given by Assumption \ref{assum1} and $\bar L$ be the operator with the coefficients $\bar a^{ij}$. Then we have
$$
-u_t+\bar Lu-\lambda u=\bar f
$$
in $\bR \times \bR^d_+$, where $\bar f=f+(\bar a^{ij}-a^{ij})D_{ij}u$. It follows from Lemma \ref{lem4.7} that the left-hand side of \eqref{osc2} is less than
$$
N\kappa^{d+2} \dashint_{Q^{+}_{\kappa r}
(t_0,x_0)} |f+(\bar a^{ij}-a^{ij})D_{ij}u|^q \,dx\,dt
 +N\kappa^{-q\alpha} \dashint_{Q^{+}_{\kappa r}(t_0,x_0)}
 | U|^q\,dx\,dt.
$$
By H\"older's inequality,
$$
\dashint_{Q^{+}_{\kappa r}
(t_0,x_0)} |(\bar a^{ij}-a^{ij})D_{ij}u|^q \,dx\,dt=\dashint_{Q^{+}_{\kappa r}
(t_0,x_0)} |1_Q(\bar a^{ij}-a^{ij})D_{ij}u|^q \,dx\,dt
$$
$$
\le \Big(\dashint_{Q^{+}_{\kappa r}(t_0,x_0)}|  D^2 u|^{\beta q}\,dx\,dt\Big)^{\frac 1 {\beta}}
\Big(\dashint_{Q^{+}_{\kappa r}(t_0,x_0)}|1_Q(\bar a^{ij}-a^{ij})|^{\beta' q}\,dx\,dt\Big)^{\frac 1 {\beta'}}
$$
$$
\le 2^{\frac{1}{\beta'}}\Big(\dashint_{Q^{+}_{\kappa r}(t_0,x_0)}|  D^2 u|^{\beta q}\,dx\,dt\Big)^{\frac 1 \beta}
\Big(\dashint_{Q}|(\bar a^{ij}-a^{ij})|^{\beta' q}\,dx\,dt\Big)^{\frac 1 {\beta'}}
$$
$$
\le N\Big(\dashint_{Q^{+}_{\kappa r}(t_0,x_0)}
|  D^2 u|^{\beta q}\,dx\,dt\Big)^{\frac 1 \beta}\gamma^{\frac 1 {\beta'}},
$$
where the last inequality is due to Assumption \ref{assum1}. Thus
collecting the above inequalities we get \eqref{osc2} immediately.
The lemma is proved.
\end{proof}

\begin{corollary}
                                \label{cor4.7}
Let $\lambda\ge 0$, $p\in (1,\infty)$, $(t_1,x_1)\in \bR\times \overline{\bR^d_+}$, and
$f\in L_p(\bR\times\bR^d_+)$. Suppose $u\in W^{1,2}_{p}(\bR\times\bR^d_+)$ vanishes
outside $Q^+_{R_0}(t_1,x_1)$ and satisfies
\begin{equation*}
-u_t+Lu-\lambda u=f
\end{equation*}
in $\bR\times\bR^d_+$ with the Dirichlet boundary condition $u=0$ on $\bR\times\partial\bR_+^d$.
Then there exists a constant $\alpha_1=\alpha_1(p)>0$ such that
under Assumption \ref{assum1} the following holds. For any
$\gamma\in (0,1)$, we have
\begin{multline}
                                        \label{eq18.16.59}
\|u_t\|_{L_p(\bR\times\bR^d_+)}+\|DD_{x''} u\|_{L_p(\bR\times\bR^d_+)}+\lambda \|u\|_{L_p(\bR\times\bR^d_+)}
\\
\le
N\gamma^{\alpha_1} \|D^2
u\|_{L_p(\bR\times\bR^d_+)}+N_1\|f\|_{L_p(\bR\times\bR^d_+)},
\end{multline}
where $N=N(d,\delta, p)>0$ and $N_1=N_1(d, \delta, p, \gamma)>0$. The
same estimate holds for $p\in (2,\infty)$ if $D_1u$ vanishes on
$\bR\times\partial\bR_+^d$ instead of $u$.
\end{corollary}
\begin{proof}
We take $q\in (1,2]$ and $\beta\in (1,\infty)$ such that $p>\beta q$. Due to \eqref{eq1213}, \eqref{eq1214}, and Lemma \ref{lem4.8}, we obtain a pointwise estimate
\begin{multline}
                                \label{eq18.16.36}
(U)^\#(t_0,x_0)
\le N\kappa^{\frac {d+2} q}\big(\cM (|f|^q(t_0,x_0))\big)^{\frac 1 q}
+N\kappa^{-\alpha}\big(\cM (|U|^q)(t_0,x_0)\big)^{\frac 1 q}\\
+N\kappa^{\frac {d+2} q}\gamma^{\frac 1 {\beta' q}}
\big(\cM (|D^2 u|^{\beta q})(t_0,x_0)\big)^{\frac 1 {\beta q}}
\end{multline}
for any $(t_0,x_0)\in \bR\times\overline{\bR^d_+}$.
As in the proof of Proposition \ref{prop5.8}, we deduce from \eqref{eq18.16.36} that
$$
\|U\|_{L_p(\bR\times\bR^d_+)}
\le N\kappa^{\frac {d+2} q}\|f\|_{L_p(\bR\times\bR^d_+)}+N\kappa^{-\alpha}
\|U\|_{L_p(\bR\times\bR^d_+)}
+N\kappa^{\frac {d+2} q}\gamma^{\frac 1 {\beta' q}}\|D^2 u\|_{L_p(\bR\times\bR^d_+)}.
$$
By taking $\kappa$ sufficiently large such that $N\kappa^{-\alpha}\le 1/2$, we get
\begin{equation}
                                  \label{eq18.16.38}
\|U\|_{L_p(\bR\times\bR^d_+)}
\le N\|f\|_{L_p(\bR\times\bR^d_+)}
+N\gamma^{\frac 1 {\beta' q}}\|D^2 u\|_{L_p(\bR\times\bR^d_+)}.
\end{equation}
By the definition of $U$, to prove \eqref{eq18.16.59} it remains to estimate $\|D_{x_1x''}u\|_{L_p(\bR \times \bR^d_+)}$ and $\|D_{x_dx''}u\|_{L_p(\bR \times \bR^d_+)}$. To this end, we
 observe that for any $\varepsilon>0$ and each $t\in \bR$
\begin{multline}
                                                    \label{3.57}
\|D_{x_1x''}u\|_{L_{p}(\bR^d_+)}+\|D_{x_dx''}u\|_{L_{p}(\bR^d_+)}\\
\le \varepsilon\big(\|D_{x_1}^2 u\|_{L_{p}(\bR^d_+)}+\|D_{x_d}^2 u\|_{L_{p}(\bR^d_+)}\big)
+N(d,p)\varepsilon^{-1}\|D_{x''}^2u\|_{L_{p}(\bR^d_+)},
\end{multline}
which is deduced from
\begin{multline*}
\|D_{x_1x''}u\|_{L_{p}(\bR^d_+)}+\|D_{x_dx''}u\|_{L_{p}(\bR^d_+)}\le N
\|\Delta u\|_{L_{p}(\bR^d_+)}\\
\leq N\|D_{x_1}^2u\|_{L_{p}(\bR^d_+)}+N\|D_{x_d}^2u\|_{L_{p}(\bR^d_+)}+N \|D_{x''}^2u\|_{L_{p}(\bR^d_+)}
\end{multline*}
by scaling in $x'' =(x_2,\ldots,x_{d-1})$. Combining \eqref{eq18.16.38} and \eqref{3.57},
we reach \eqref{eq18.16.59} upon choosing
$\varepsilon=\gamma^{\frac 1 {2\beta' q}}$. The last assertion follows from the
last assertion of Lemma \ref{lem4.8} by using the same proof.
\end{proof}

\subsection{Estimates for divergence form operators}

Let
$$
\cL u = D_i (a^{ij} D_j u ),
$$
where $a^{ij}$ satisfy Assumption \ref{assum1}.

Following the proof of Lemma \ref{lem4.8}, we derive the following lemma from Lemmas \ref{lem5.6} and \ref{lem5.34} with $k=1$.

\begin{lemma}
                                        \label{lem4.4}
Let $\cL\in \bL_2$, $\lambda>0$, $\beta\in (1,\infty)$ and
$\beta'=\beta/(\beta-1)$ be constants, $(t_1,x_1)\in \bR\times\overline{\bR^d_+}$,
and $f=(f_1,\ldots,f_d), g \in
C_{\text{loc}}^\infty(\bR\times\overline{\bR^d_+})$. Suppose that $u\in C_{\text{loc}}^\infty(\bR\times\overline{\bR^d_+})$ vanishes outside $Q^+_{R_0}(t_1,x_1)$ and satisfies
$$
-u_t+\cL u-\lambda u=\Div f +g
$$
locally in $\bR\times\bR^d_+$ with the conormal derivative boundary condition $a^{11}D_1 u=f$
on $\bR\times\partial\bR^d_+$. Then under Assumption \ref{assum1}, for any
$r>0$, $\kappa\ge 32$, and $(t_0,x_0)\in \bR\times\overline{\bR^d_+}$, we have
\begin{multline}
                                                  \label{eq17.21.22z}
\dashint_{Q^{+}_r(t_0,x_0)}\dashint_{Q^{+}_r(t_0,x_0)}|D_1u(t,x)-
D_1u(s,y)|^2\,dx\,dt\,dy\,ds\\
\leq N\kappa^{d+2} \dashint_{Q^{+}_{\kappa r} (t_0,x_0)}
| h|^2\,dx\,dt
 +N\kappa^{-2\alpha} \dashint_{Q^{+}_{\kappa
 r}(t_0,x_0)}|D_1u|^2\,dx\,dt\\
 +N\kappa^{d+2} \Big(\dashint_{Q^{+}_{\kappa r}(t_0,x_0)}
 |Du|^{2\beta}\,dx\,dt\Big)^{1/\beta}\gamma^{1/\beta'},
\end{multline}
where $h = |f| + \lambda^{-1/2}|g|$, $\alpha = \alpha(d,\delta) \in (0,1)$, and $N=N(d,\delta)>0$.
The same estimates holds for $\cL\in \bL_1$ if the conormal
derivative boundary condition is replaced with the Dirichlet
boundary condition $u=0$ on $\bR\times\partial\bR_+^d$.
\end{lemma}

Next we derive an a priori estimate for solutions to divergence form equations.

\begin{proposition}
                                \label{prop4.4}
Suppose either $\cL\in \bL_1$ and $p\in (1,2]$ or $\cL\in \bL_2$
and $p\in [2,\infty)$.
Let $f=(f_1,\ldots,f_d), g \in
L_p(\bR \times \bR^d_+)\cap C_{\text{loc}}^\infty(\bR\times \overline{\bR^d_+})$. Then there exist constants $\gamma\in (0,1)$ and $N>0$ depending only on $d$, $\delta$ and $p$, and $\lambda_0\ge 0$ depending only on these parameters as well as $R_0$,  such
that under Assumption \ref{assum1} the following holds true.
For any $\lambda > \lambda_0$ and $u\in \cH^1_p(\bR \times \bR^d_+)\cap C_{\text{loc}}^\infty(\bR\times \overline{\bR^d_+})$ satisfying
$$
-u_t+\cL u-\lambda u=\Div f  + g
$$
in $\bR\times\bR^d_+$ with the conormal derivative boundary condition $a^{1j}D_ju=f_1$
on $\bR\times\partial\bR^d_+$, we have
\begin{equation}
                \label{eq15.16.35}
\lambda^{1/2}\|u\|_{L_{p}(\bR\times\bR^d_+)}+
\|Du\|_{L_{p}(\bR\times\bR^d_+)}\le N\|f\|_{L_{p}(\bR\times\bR^d_+)}+ N \lambda^{-1/2}\|g\|_{L_{p}(\bR\times\bR^d_+)}.
\end{equation}
\end{proposition}
\begin{proof}
The case $p=2$ follows from Theorem \ref{thm0411_1}.
By the duality argument, we may assume that $p\in (2,\infty)$ and $\cL\in \bL_2$.
First we consider the case when $u$ vanishes outside $Q_{R_0}^+(t_1,x_1)$ for some $(t_1,x_1)\in \bR \times \overline{\bR^d_+}$.
Following the proof of Proposition \ref{prop5.8}, from \eqref{eq17.21.22z} we obtain
\begin{align*}
\|D_1 u\|_{L_p(\bR\times\bR^d_+)}&\le N\kappa^{\frac {d+2} 2}\|h\|_{L_p(\bR\times\bR^d_+)}\\
&\,\,+N\kappa^{-\alpha}\|D_1 u\|_{L_p(\bR\times\bR^d_+)}+N\kappa^{\frac {d+2} 2}\gamma^{\frac 1 {2\beta'}}\|D u\|_{L_p(\bR\times\bR^d_+)},
\end{align*}
which implies
\begin{equation}
                                    \label{eq11.24}
\|D_1 u\|_{L_p(\bR\times\bR^d_+)}\le N\|h\|_{L_p(\bR\times\bR^d_+)}+N\gamma^{\frac 1 {2\beta'}}\|D u\|_{L_p(\bR\times\bR^d_+)},
\end{equation}
upon choosing $\kappa$ sufficiently large such that $N\kappa^{-\alpha}\le 1/2$. Now since $\cL\in \bL_2$, we can rewrite the equation into
\begin{equation}
                                \label{eq1.04}
-u_t+\tilde \cL u-\lambda u=\Div \tilde f+g,
\end{equation}
where
\begin{align*}
\tilde \cL u&=D_1(a^{11}D_1 u)+\sum_{i,j=2}^d D_i(a^{ij}D_j u),\\
\tilde f_1&=f_1,\quad \tilde f_i=f_i-a^{i1}D_1 u,\quad i=2,\ldots,d.
\end{align*}
We take the even extensions of $u$, $a^{11}$, $a^{ij},i,j\ge 2$, $\tilde f_i,i\ge 2$, and $g$ with respect to $x_1$, and the odd extension of $\tilde f_1$ with respect to $x_1$. It is easily seen that after these extensions, $u\in \cH^1_p(\bR\times\bR^d)$ satisfies \eqref{eq1.04} in $\bR\times\bR^d$, and the coefficients of $\tilde \cL$ satisfies Assumption \ref{assum1} with $2\gamma$ in place of $\gamma$. Thanks to the $L_p$-estimates for divergence type equations in the whole space with partially BMO coefficients (see Theorem 6.3 in \cite{DK11b}), we get
\begin{multline}
                                        \label{eq1.10}
\lambda^{1/2}\|u\|_{L_{p}(\bR\times\bR^d)}+
\|Du\|_{L_{p}(\bR\times\bR^d)}\le N\|\tilde f\|_{L_{p}(\bR\times\bR^d)}+ N \lambda^{-1/2}\|g\|_{L_{p}(\bR\times\bR^d)}\\
\le N\|f\|_{L_{p}(\bR\times\bR^d_+)}+N\|D_1 u\|_{L_{p}(\bR\times\bR^d_+)}+ N \lambda^{-1/2}\|g\|_{L_{p}(\bR\times\bR^d_+)}
\end{multline}
provided that $\gamma<\gamma_0(d,\delta,p)$. Combining \eqref{eq11.24} and \eqref{eq1.10} and choosing $\gamma$ even smaller, we obtain  \eqref{eq15.16.35} for any $\lambda>0$ and $u$ vanishing outside $Q_{R_0}^+(t_1,x_1)$. To complete the proof of \eqref{eq15.16.35} for general $u\in \cH^1_p(\bR \times \bR^d_+)\cap C_{\text{loc}}^\infty(\bR\times \overline{\bR^d_+})$, we use the standard partition of unity argument. See, for instance, the proof of Theorem 5.7 in \cite{Kr07}. The proposition is proved.
\end{proof}

\begin{remark}
Theorem 6.3 in \cite{DK11b} requires that the function $\vu$ vanishes outside $Q_{\gamma R_0}(t_1,x_1)$ instead of $Q_{R_0}(t_1,x_1)$.
Nevertheless, the theorem still holds with $Q_{R_0}$ by modifying the second part (the case $\kappa r \ge R_0$) of the proof of Theorem 6.1 in the same paper. Indeed, in the proof of the case $\kappa r \ge R_0$, we find $\cT_Q \in \bO$ and $\{ \bar a_{\alpha\beta} \}_{|\alpha|=|\beta|=m} \in \bA$ for $Q = Q_{R_0}$, and proceed as in the proof of Lemma \ref{lem4.8} in this paper.
\end{remark}

Similarly, we deduce the following proposition from the second assertion of Lemma \ref{lem4.4}.
\begin{proposition}
                                \label{prop4.5}
Suppose either $\cL\in \bL_2$ and $p\in (1,2]$ or $\cL\in \bL_1$
and $p\in [2,\infty)$. Let $f=(f_1,\ldots,f_d), g \in
L_p(\bR \times \bR^d_+)\cap C_{\text{loc}}^\infty(\bR\times \overline{\bR^d_+})$. Then there exist constants $\gamma \in (0,1)$ and $N>0$ depending only on $d$, $\delta$ and $p$, and $\lambda_0\ge 0$ depending only on these parameters as well as $R_0$, such
that under Assumption \ref{assum1} the following holds true. For
any $\lambda > \lambda_0$ and $u\in \cH^1_p(\bR \times \bR^d_+) \cap C_{\text{loc}}^\infty(\bR\times \overline{\bR^d_+})$ satisfying
$$
-u_t+\cL u-\lambda u=\Div f + g
$$
in $\bR\times\bR^d_+$ with the Dirichlet boundary condition $u= 0$
on $\bR\times\partial\bR^d_+$, we have
\begin{equation*}
\lambda^{1/2}\|u\|_{L_{p}(\bR\times\bR^d_+)}+
\|Du\|_{L_{p}(\bR\times\bR^d_+)}\le N\|f\|_{L_{p}(\bR\times\bR^d_+)} + N \lambda^{-1/2}\|g\|_{L_{p}(\bR\times\bR^d_+)}.
\end{equation*}
\end{proposition}

\subsection{Proofs of Theorems \ref{thm3} and \ref{thm4}}

\begin{proof}[Proof of Theorem \ref{thm3}]
As in the proof of Theorem \ref{thm01}, it suffices us to prove the a priori estimate \eqref{eq15.16.02} when $T = \infty$ assuming that $u \in C_0^{\infty}(\bR \times\overline{\Xi^k})$, $f \in C_0^{\infty}(\bR \times \overline{\Xi^k})$, and $a^{ij}$ are infinitely differentiable.
At the moment, we also assume that $b^i \equiv c \equiv 0$ and $u$ vanishes outside $B_{R_0}^+(t_1,x_1)$ for some $(t_1,x_1) \in \bR \times \overline{\bR^d_+}$.
Then from Corollary \ref{cor4.7} we obtain
\begin{multline}
							\label{eq0531_3}
\|u_t\|_{L_p(\bR\times\bR^d_+)}+\|DD_{x''} u\|_{L_p(\bR\times\bR^d_+)}+\lambda \|u\|_{L_p(\bR\times\bR^d_+)}
\\
\le N\gamma^{\alpha_1} \|D^2 u\|_{L_p(\bR\times\bR^d_+)}+N_1\|f\|_{L_p(\bR\times\bR^d_+)},
\end{multline}
where $N=N(d,\delta,p)$ and $N_1 = N_1(d,\delta,p,\gamma)$.

To apply the estimates for divergence equations obtained in the previous subsection, we divide both sides of the equation \eqref{eq15.16.01} by $a^{11}$, then add $-u_t$, $\Delta_{x''}u = D_{x_2x_2}u + \ldots + D_{x_{d-1}x_{d-1}}u$, and $- \lambda u$. We then get
\begin{multline}
							\label{eq0531_1}
-u_t + \Delta_{x'} u + \frac{a^{1d} + a^{d1}}{a^{11}} D_{d1} u + \frac{a^{dd}}{a^{11}} D_{dd} u - \lambda u
\\
= (1/a^{11} - 1) u_t + \lambda (1/a^{11} - 1) u + f/a^{11} + \Delta_{x''} u - \sum_{i,j=1}^d \hat a^{ij} D_{ij} u,
\end{multline}
where $\hat a^{ij} = a^{ij}/a^{11}$ for all $i, j = 1, \ldots, d$, except the following terms:
$$
\hat a^{11} = \hat a^{1d} = \hat a^{d1}  = \hat a^{dd} = 0.
$$
Set $\tilde a^{d1} = (a^{1d} + a^{d1})/a^{11}$, $\tilde a^{dd} = a^{dd}/a^{11}$, and
$$
\tilde f = (1/a^{11} - 1) u_t + \lambda (1/a^{11} - 1) u + f/a^{11} + \Delta_{x''} u  - \sum_{i,j=1}^d \hat a^{ij} D_{ij} u.
$$
Then the equation \eqref{eq0531_1} turns into
\begin{equation}
							\label{eq0531_2}
-u_t + \Delta_{x'} u + \tilde a^{d1} D_{d1} u + \tilde a^{dd} D_{dd} u - \lambda u = \tilde f
\end{equation}
in $\bR \times \bR^d_+$.
Denote $w = D_d u$
and
$$
\cL w = \Delta_{x'} w + D_d(\tilde a^{d1} D_1 w) + D_d ( \tilde a^{dd} D_d w ).
$$
By differentiating the equation \eqref{eq0531_2} in $x_d$, we see that $w$ satisfies the following divergence type equation
\begin{equation}
							\label{eq0604_6}
-w_t + \cL w - \lambda w
= D_d \tilde f
\end{equation}
with the Dirichlet boundary condition $w = 0$ on $\bR \times \partial \bR^d_+$.
Notice that $\cL \in \bL_2$ and the coefficients of $\cL$ satisfy Assumption \ref{assum1} with $N(\delta) \gamma$ in places of $\gamma$.
Then by Proposition \ref{prop4.5} there exist $\gamma_1 \in (0,1)$ and $N$, depending only on $d, \delta, p$, and $\lambda_0\ge 0$ depending only on these parameters as well as $R_0$, such that under Assumption \ref{assum1} with any $\gamma\in (0,\gamma_1]$ and the condition $\lambda > \lambda_0$, we have
\begin{equation}
							\label{eq0531_4}
\|DD_du\|_{L_p(\bR \times \bR^d_+)}
= \|Dw\|_{L_p(\bR \times \bR^d_+)}
\le N \|\tilde f\|_{L_p(\bR \times \bR^d_+)}.
\end{equation}
On the other hand, by the definition of $\tilde f$ and \eqref{eq0531_3},
\begin{align*}
\|\tilde f\|_{L_p(\bR\times\bR^d_+)}
&\le N \|u_t\|_{L_p(\bR\times\bR^d_+)}
+ N \|DD_{x''}u\|_{L_p(\bR\times\bR^d_+)}\\
&\quad+ N \lambda \|u\|_{L_p(\bR\times\bR^d_+)}
+ N \|f\|_{L_p(\bR\times\bR^d_+)},
\end{align*}
where $N=N(d,\delta)$.
This combined with \eqref{eq0531_3} and \eqref{eq0531_4} gives
\begin{align*}
&\|u_t\|_{L_p(\bR\times\bR^d_+)}+\sum_{ij > 1}\|D_{ij} u\|_{L_p(\bR\times\bR^d_+)}+\lambda \|u\|_{L_p(\bR\times\bR^d_+)}\\
&\quad\le N\gamma^{\alpha_1} \|D^2 u\|_{L_p(\bR\times\bR^d_+)}+N_1\|f\|_{L_p(\bR\times\bR^d_+)}.
\end{align*}
The missing term $D_1^2u$ from the left-hand side of the above inequality can be estimated by the above estimate and the equation itself. Therefore, we obtain
\begin{align*}
&\|u_t\|_{L_p(\bR\times\bR^d_+)}+\|D^2 u\|_{L_p(\bR\times\bR^d_+)}+\lambda \|u\|_{L_p(\bR\times\bR^d_+)}\\
&\quad\le N\gamma^{\alpha_1} \|D^2 u\|_{L_p(\bR\times\bR^d_+)}+N_1\|f\|_{L_p(\bR\times\bR^d_+)}.
\end{align*}
Now we choose $\gamma$ small enough so that $\gamma \le \gamma_1$ and $N\gamma^{\alpha_1} < 1/2$.
Then using the interpolation argument to obtain $\lambda^{1/2}\|Du\|_{L_p(\bR \times \bR^d_+)}$, we finally obtain \eqref{eq15.16.02}

Next, using the partition of unity argument, moving the $b^i D_i u$ and $cu$ terms to the right-hand side of the equation, taking $\lambda_0$ even larger, and using interpolation inequalities, we remove the restrictions that $b^i \equiv c \equiv 0$ and $u$ vanishes outside $B^+_{R_0}(t_1, x_1)$ in the previous step.
The theorem is proved.
\end{proof}

\begin{proof}[Proof of Theorem \ref{thm4}]
The proof repeats the same lines in the proof of Theorem \ref{thm3} except that the function $w: = D_d u$ satisfies \eqref{eq0604_6} with the conormal derivative condition $D_1 w = 0$ on $\bR \times \partial \bR^d_+$, and we use Proposition \ref{prop4.4} instead of Proposition \ref{prop4.5} and the second assertion of Corollary \ref{cor4.7} instead of the first one. We omit the details.
\end{proof}

\bibliographystyle{plain}

\end{document}